\documentclass[12pt]{article}

\usepackage[latin1]{inputenc}
\usepackage[english]{babel}
\usepackage{amssymb,amsmath}
\usepackage{amsthm}
\usepackage{setspace,lipsum}
\usepackage{amsfonts}
\usepackage{graphicx}
\usepackage[margin=2.4cm]{geometry}
\usepackage{enumerate}
\usepackage{enumitem}
\usepackage{color}
\usepackage[noadjust]{cite}
\usepackage{tikz}
\usepackage{mathtools}
\usepackage{hyperref}
\usepackage{multirow}
\usepackage{array}

\newtheorem{thm}{Theorem}

\newtheorem{cor}[thm]{Corollary}

\newtheorem{lem}[thm]{Lemma}

\theoremstyle{definition}

\newtheorem{defn}{Definition}

\newtheorem*{conj*}{Conjecture}

\newcommand{\mymod}{{\ \textrm{mod} \ }}

\usepackage{xfrac}
\usepackage{xcolor}

\usepackage{tikz}

\newcommand{\Path}[1]{\ensuremath{{\sf P}_{#1}}}

\newcommand{\C}[1]{\ensuremath{{\sf C}_{#1}}}

\newcommand{\K}[1]{\ensuremath{{\sf K}_{#1}}}

\newcommand{\textoverline}[1]{$\overline{\mbox{#1}}$}

\title{Extremal chemical graphs of maximum degree at most 3 for 33 degree-based topological indices}
\author{
	S\'ebastien Bonte\textsuperscript{1},
	Gauvain Devillez\textsuperscript{1},
	Valentin Dusollier\textsuperscript{1},\\
	Alain
	Hertz\textsuperscript{2},
	Hadrien M\'elot\textsuperscript{1}	
	\\[3mm]
	\footnotesize \textsuperscript{1} Computer Science Department - Algorithms Lab\\
\footnotesize University of Mons, Mons, Belgium\\[3mm]	
\footnotesize \textsuperscript{2} Department of Mathematics and Industrial
	Engineering\\
	\footnotesize Polytechnique Montr\'eal - Gerad, Montr\'eal, Canada\\
	\footnotesize Corresponding author. Email: alain.hertz@gerad.ca
}
\date{}

\begin{document}

\maketitle

\vspace*{0.2cm}

\hrule
\vspace*{0.2cm}
\small
\noindent
\textbf{Abstract.} \emph{We consider chemical graphs that are defined as connected graphs of maximum degree at most 3. We characterize the extremal graphs, meaning those that maximize or minimize 33 degree-based topological indices. This study shows that five graph families are sufficient to characterize the extremal graphs of 29 of these 33 indices. In other words, the extremal properties of this set of degree-based topological indices vary very little.}

\vspace*{0.2cm}

\vspace*{0.2cm}
\hrule

\normalsize

\section{Introduction} \label{sec_intro}

Chemical graphs provide a powerful tool for modeling molecular structures. These graphs, where vertices represent atoms and edges represent bonds, allow researchers to investigate various chemical and physical properties of molecules through graph-theoretical concepts. 

According to Patrick Fowler \cite{PF}: ``The definition of chemical graphs that is useful depends on context. Two definitions appropriate to different kinds of carbon framework can be found in the literature. The graphs that can be regarded as skeletons of saturated hydrocarbons (such as alkanes), are connected and have maximum degree $\Delta$~at most~4. If instead the interest is in (unsaturated) conjugated $\pi$ systems, such as alkenes, polyenes, benzenoids, and fullerenes, the maximum degree should be at most~3, since a conjugated carbon atom participates in at most three single bonds.''

In this paper, we focus on the second definition of chemical graphs (where the maximum degree is at most~3) and explore the bounds on topological indices of such graphs. A topological index, or molecular descriptor, is a graph invariant used to study specific physicochemical properties of molecules. Among the most well-known indices is the Randi\'c index, introduced by Milan Randi\'c~\cite{Randic1975} in 1975, which has been widely used in quantitative structure-activity relationship (QSAR) and quantitative structure-property relationship (QSPR) studies. Its value $Ra(G)$ for a chemical graph $G$ is defined as 
$$
Ra(G) = \sum_{vw \in E} \frac{1}{\sqrt{d(v)  d(w)}},
$$
where $d(v)$ is the degree of vertex $v$. This is an example of so-called \emph{degree-based} topological index, that is an index computed from the sum of the weights of the edges, each edge $vw$ having a weight defined by a formula using the degrees of $v$ and $w$.

As stated by Ivan Gutman \cite{G13}, ``Countless topological indices have been and are being proposed
so far, in many cases without any examination if these
correlate with any of the various physical properties,
chemical reactivity or biological activity. To use a mild expression, today we have far too
many such descriptors, and there seems to lack a firm
criterion to stop or slow down their proliferation.''

In this paper, we consider 33 degree-based topological indices that we found in the literature (see Section~\ref{sec_preliminaries})
and whose extremal properties have given rise to scientific publications \cite{das2011abc,zhang2016abc,AFRG22,henning2007albertson,hansen2005albertson,cui2021ag,carballosa2022ag,VPVFS21,che2016forgotten,yan2010ga,deng2018ga,zhong2012harmonic, DENG2013,elumalai2018, das2016,sedlar2015, falahati2017,ali2023maximum,cruz2021sombor,li2022extremal,cruz2021extremal,liu2021reduced,deng2021molecular,ali2019extremal,gutman2004first,nikolic2003zagreb,furtula2010augmented, ali2021augmented,li2008survey, swartz2022survey, hansen2009variable,GFE14,liu2020some, chu2020extremal,raza2020bounds,ali2020sddi,ghorbani2021sddi}. In the same spirit as Gutman's  words, we can  wonder whether these indices are very different from each other. 
We provide a partial answer by analyzing the extremal properties of these  indices.
We use the word ``partial'' for several reasons. First, we are only interested in the \emph{extremal} properties of topological indices and it could therefore be that various indices are distinguished by other properties of interest to chemists. Second, we only deal with chemical graphs of maximum degree at most 3. Finally, the list of topological indices studied in this article is not exhaustive, although we have tried to consider the most cited and studied in the scientific literature.
Our conclusions will be clear: five families of chemical graphs are sufficient to characterize the vast majority of extremal chemical graphs of degree-based topological indices.

Let $G = (V, E)$ be a graph of {\em order} $n = |V|$ and {\em size} $m = |E|$. The maximum degree of a graph $G$ is denoted $\Delta(G)$. An edge with endpoints $u$ and $v$ of degree $d(u)=i$ and $d(v)=j$ is called an $(i,j)$-edge and is denoted $uv$. 
 We denote $x_{ij}$ the number of $(i,j)$-edges in $G$ while $n_i$ is the number of vertices in $G$ of degree $i$. In what follows,  $\K{n}$, $\Path{n}$ and $\C{n}$ denote the complete graph of order $n$, the path of order $n$ and the cycle of order $n$, respectively.

In the next section, we give a precise definition of the chemical graphs considered in this paper and we give the list of 33 topological indices whose extremal properties are analyzed. Section 3 is dedicated to defining five families of chemical graphs which are sufficient to characterize the extremal graphs for a large majority of degree-based topological indices. Tools used in our proofs are given in Section 4, and a characterization of extremal chemical graphs for the 33 topological indices is given in Section 5.

\section{Preliminaries}\label{sec_preliminaries}

As mentioned in the previous section, we are interested in connected graphs of maximum degree at most 3. To avoid border effects, we will not consider small or dense graphs which have only few possible $x_{ij}$ values. This is now explained in detail. 

There are only 10 connected graphs of order $n$ with $1\leq n \le 4$, Six of them, namely \K{1},\K{2},\K{3}, \K{4}, \Path{3} and the diamond (\K{4} minus an edge), are the only ones having their order and size. They therefore maximize and minimize any topological index of their order and size. The two pairs $(n,m)$ with $n\leq 4$ that have different chemical graphs of order $n$ and size $m$ are \Path{4} and the star with 3 branches for $(n,m)=(4,3)$ and \C{4} and a triangle plus a pending vertex for $(n,m)=(4,4)$.
By restricting ourselves to connected graphs of maximum degree at most 3, it is not difficult to show that there are 10 such graphs of order $n=5$ and 29 ones of order $n=6$. These can be obtained
using \emph{PHOEG}~\cite{phoeg}, \emph{House of Graphs}~\cite{hog} or \emph{Nauty's geng}~\cite{geng}.
Hence, given any topological index, it is easy to determine which chemical graph of order $n\leq 6$ has maximum or minimum value. From now on, we will therefore only consider connected graphs $G$ of order at least 7, which implies $x_{11}=0$ and $2\leq \Delta(G)\leq 3$.

\begin{defn}
A \emph{degree-based topological index} is any function $f$ of the form
\vspace{-0.2cm}$$f(x_{12},x_{13},x_{22},x_{23},x_{33}) = c_{12}x_{12}+c_{13}x_{13}+c_{22}x_{22}+c_{23}x_{23}+c_{33}x_{33},$$

\vspace{-0.3cm}where every $c_{ij}$ is a real number.
\end{defn} 

By abuse of notation, for a graph $G$, we will write $f(G)$ instead of $f(x_{12},x_{13},x_{22},x_{23},x_{33})$, where $x_{ij}$ is the number of $(i,j)$-edges in $G$.
For example, the Randi\'c index (see Section~\ref{sec_intro}) is the degree-based topological index with $c_{ij} = \frac{1}{\sqrt{i j}}$. 

Let's focus now on dense graphs. Since we restrict ourselves  to graphs $G$ of maximum degree at most 3, the size $m$ of such graphs is at most $\frac{3n}{2}$: if $m=\frac{3n}{2}$ , then $x_{33}=m$ ($G$ is 3-regular); if $m=\frac{3n-1}{2}$, then $x_{23}=2$ and $x_{33}=m-2$; if $m=\frac{3n-2}{2}$, then there are three possible cases:
\begin{itemize}[nosep]
	\item $x_{13}=1$ and $x_{33}=m-1$;
	\item $x_{23}=4$ and $x_{33}=m-4$;
	\item $x_{22}=1$, $x_{23}=2$ and $x_{33}=m-3$.
\end{itemize}
Hence, given a pair $(n,m)$ with $m\geq \frac{3n-2}{2}$, and given any degree-based topological index $f$, it is not difficult to determine the $x_{ij}$ values of the connected graphs of order $n$, size $m$ and maximum degree at most 3 which maximize or minimize $f$.

From now on, when we talk about chemical graphs, we assume that we are not in the above extreme cases (i.e., very small or very dense graphs). More precisely, here is the definition of the chemical graphs studied in this paper.

\begin{defn}
A \emph{chemical graph} is a connected graph of order $n\geq 7$,  size $m\leq \frac{3n-3}{2}$ and maximum degree at most 3. 
\end{defn}

It is important to specify here that although the results that we demonstrate are valid for chemical graphs as defined above, it is possible that these results are also valid for some connected graphs of maximum degree at most 3 and of order $n<7$ or size $m>\frac{3n-3}{2}.$

\begin{table} 
	\begin{center}
		\caption{33 Degree-based topological indices} \label{tab_inv}
		\begin{tabular}{| l | l | l |}
			\hline
			Name & Short name & $c_{ij}$ \\
			\hline
			Atom-bond connectivity index & ABC & $\sqrt{\frac{i+j-2}{ij}}$ \\
			Atom-bond sum-connectivity index & ABSC & $\sqrt{\frac{i+j-2}{i+j}}$ \\
			Albertson index & Albertson & $|i-j|$ \\
			Arithmetic-geometric index & AG & $\frac{i+j}{2\sqrt{ij}}$ \\
			Difference between AG and GA & AG-GA & $\frac{i+j}{2\sqrt{ij}} - \frac{2\sqrt{ij}}{i+j}$ \\
			Extended index & Extended & $\frac{1}{2}(\frac{i}{j} + \frac{j}{i})$ \\
			Forgotten index & Forgotten & $i^2 + j^2$ \\
			Geometric-arithmetic index & GA & $\frac{2\sqrt{ij}}{i+j}$ \\
			First Gourava index & Gourava1 & $i+j+ij$ \\
			Second Gourava index & Gourava2 & $(i+j)ij$ \\
			First hyper-Gourava index & hGourava1 & $(i+j+ij)^2$ \\
			Second hyper-Gourava index & hGourava2 & $((i+j)ij)^2$ \\
			Gourava sum-connectivity index & GouravaSC & $\frac{1}{\sqrt{i+j+ij}}$ \\
			Gourava product-connectivity index & GouravaPC & $\sqrt{ij(i+j)}$\\
			Harmonic index & Harmonic & $\frac{2}{i+j}$ \\
			Inverse degree index & InvDeg & $i^{-2} + j^{-2}$ \\
			Inverse sum of degree index & InvSumDeg & $\frac{ij}{i+j}$ \\
			Randi\'c index & Randi\'c& $\frac{1}{\sqrt{ij}}$ \\
			Reciprocal Randi\'c index & rRandi\'c & $\sqrt{ij}$ \\
			Sigma index & Sigma & $(i-j)^2$ \\
			Sombor index & Sombor & $\sqrt{i^2 + j^2}$ \\
			Reduced Sombor index & rSombor & $\sqrt{(i-1)^2 + (j-1)^2}$ \\
			Sum connectivity index & SumConn & $\frac{1}{\sqrt{i+j}}$ \\
			Reciprocal sum connectivity index & rSumConn & $\sqrt{i+j}$ \\
			First Zagreb index & Zagreb1 & $i+j$ \\
			Second Zagreb index & Zagreb2 & $ij$ \\
			Augmented Zagreb index & aZagreb & $(\frac{ij}{i+j-2})^3$\\
			First hyper-Zagreb index & hZagreb1 & $(i+j)^2$ \\
			Second hyper-Zagreb index & hZagreb2 & $(ij)^2$ \\
			Nat. log. of the mult. sum Zagreb index & lnZagreb1 & $\ln(i+j)$\\
			Nat. log. of the first mult. Zagreb index & lnZagreb2 & $2(\frac{\ln(i)}{i}+\frac{\ln(j)}{j})$\\
			Nat. log. of the second mult. Zagreb index & lnZagreb3 & $\ln(i)+\ln(j)$\\
			Modified first Zagreb index & mZagreb & $i^{-3} + j^{-3}$ \\
			\hline
		\end{tabular}
	\end{center}
\end{table}

We found in the literature 33 degree-based topological indices. They are described in  Table~\ref{tab_inv}. 
Most of them, namely 28, appear in \cite{G21}, the exceptions being ABSC which appears in \cite{AFRG22}, AG-GA, which appears in \cite{VPVFS21} and lnZagreb1, lnZagreb2 and lnZagreb3 which can be found in \cite{raza2020bounds}. 
We are interested in the extremal properties of these indices. More precisely, given a topological index $f$, we aim to characterize the  chemical graphs that maximize $f$ and those that minimize $f$. For the 33 indices of Table \ref{tab_inv}, this gives potentially 66 families of chemical graphs. As will be shown, 5 families (instead of 58) are sufficient to characterize the extremal chemical graphs of 29 of the 33 topological indices.

\begin{defn}
	Given a degree-based topological index $f$ defined by $c_{ij}$ values, its complement denoted  \textoverline{$f$} is the degree-based topological index  defined by $-c_{ij}$ values.
\end{defn}	

Determining chemical graphs with the \emph{minimum} value for $f$ is thus equivalent to determining chemical graphs with the \emph{maximum} value for \textoverline{$f$}. In the subsequent proofs, we always aim to maximize the value of a topological index in Table \ref{tab_inv} or its complement.

\begin{defn}
	A chemical graph $G$ is \emph{extremal} for a degree-based topological index $f$ if it maximizes $f$ or \textoverline{$f$} over all chemical graphs of the same order and size as $G$.
\end{defn}

\section{Five families of chemical graphs}\label{sec_4families}

A chemical graph is characterized by five $x_{ij}$ values, namely, $x_{12}, x_{13}, x_{22}, x_{23}$ and $x_{33}$. We therefore have:
\begin{align}
 \label{eq_n1}  n_1 & = x_{12} + x_{13}\\[-3pt]
 \label{eq_n2}  n_2 & = \frac{x_{12} + 2 x_{22} + x_{23} }{2}\\[-3pt] 
 \label{eq_n3}  n_3 & = \frac{x_{13} + x_{23} + 2 x_{33} }{3}\\[-3pt]
\label{n}n&=n_1+n_2+n_3=\frac{3}{2} x_{12} + \frac{4}{3} x_{13} + x_{22} + \frac{5}{6} x_{23} + \frac{2}{3} x_{33}\\[-3pt]
\label{m}m&=x_{12}+x_{13}+x_{22}+x_{23}+x_{33}.
\end{align}

We now define five families of chemical graphs. As will be shown, these are sufficient to characterize the extremal chemical graphs of 29 topological indices.

\begin{defn}\label{def_F1}
	$F_1$ is the set of chemical graphs with  the following numbers $x_{ij}$ of $(i,j)$-edges:
			\begin{center}
			\setlength{\extrarowheight}{2pt}			\begin{tabular}{| c | c | c | c |c | l}\cline{1-5} 
				$x_{12}$ & $x_{13}$ & $x_{22}$& $x_{23}$& $x_{33}$&\\ \cline{1-5}
				0 & $\frac{3n - 2m }{2}$ & 0 & 0 &$\frac{4m - 3n}{2}$& if $n$ if even\\ \cline{1-5}
				0 & $\frac{3n - 2m-1 }{2}$ & 0 & 2 &$\frac{4m - 3n-3}{2}$& if $n$ if odd\\ \cline{1-5}
			\end{tabular}
		\end{center}
\end{defn}
\begin{defn}\label{def_F2}
	$F_2$ is the set of chemical graphs with the following numbers $x_{ij}$ of $(i,j)$-edges:
	\begin{center}
	\begin{tabular}{| c | c | c | c |c | l}\cline{1-5} 
		$x_{12}$ & $x_{13}$ & $x_{22}$& $x_{23}$& $x_{33}$&\\ \cline{1-5}
		2 & 0 & $m-2$ & 0 &0& if $m=n-1$\\ \cline{1-5}
		0 & 0 & $m$ & 0 &0& if $m=n$\\ \cline{1-5}
		0 & 0 & $m-5$ & 4 &1& if $m=n+1$\\ \cline{1-5}
		0 & 0 & $3n-2m-1$ & 2 &$3m-3n-1$& if $n+1<m\leq\frac{3n-3}{2}$\\ \cline{1-5}
	\end{tabular}
\end{center}
\end{defn}
\begin{defn}\label{def_F3}
	$F_3$ is the set of chemical graphs with the following numbers $x_{ij}$ of $(i,j)$-edges:
	\begin{center}
	\setlength{\extrarowheight}{2pt}		\begin{tabular}{| c | c | c | c |c | l}\cline{1-5} 
		$x_{12}$ & $x_{13}$ & $x_{22}$& $x_{23}$& $x_{33}$&\\ \cline{1-5}
		0 & $\frac{3n - 2m }{2}$ & 0 & 0 &$\frac{4m - 3n}{2}$& if $n$ if even\\ \cline{1-5}
		1 & $\frac{3n - 2m-3 }{2}$ & 0 & 1 &$\frac{4m - 3n-1}{2}$& if $n$ if odd\\ \cline{1-5}
	\end{tabular}
\end{center}
\end{defn}
\begin{defn}\label{def_F4}
	$F_4$ is the set of chemical graphs with the following numbers $x_{ij}$ of $(i,j)$-edges:
	\begin{center}
		\begin{tabular}{| c | c | c | c |c | l}\cline{1-5} 
			$x_{12}$ & $x_{13}$ & $x_{22}$& $x_{23}$& $x_{33}$&\\ \cline{1-5}
			2 & 0 & $m-2$ & 0 &0& if $m=n-1$ \\ \cline{1-5}
			0 &0 &$6n-5m$ & $6m-6n$ &0& if $n\leq m< \frac{6n}{5}$\\ \cline{1-5}
			0 &0 & 0 & $6n-4m$ &$5m-6n$& if $\frac{6n}{5}\leq m\leq \frac{3n-3}{2}$\\ \cline{1-5}			
		\end{tabular}
	\end{center}
\end{defn}
\begin{defn}
$F_5$ is the set of chemical graphs with the following numbers $x_{ij}$ of $(i,j)$-edges:
		\begin{center}
			\begin{tabular}{| c | c | c | c |c | l}\cline{1-5} 
				$x_{12}$ & $x_{13}$ & $x_{22}$& $x_{23}$& $x_{33}$&\\ \cline{1-5}
				2 & 0 & $m-2$ & 0 &0& if $m=n-1$ \\ \cline{1-5}
				0 & 0 & $m$ & 0 &0& if $m=n$ \\ \cline{1-5}
				0 & 0 & $m-6$ & 6 &0& \multirow{2}{*}{if $m=n+1$}\\
				0 & 0 & $m-5$ & 4 &1& \\ \cline{1-5}
				0 &0 &$a$ & $6n-4m-2a$ &$5m-6n+a$& if $n+1< m\leq \frac{3n-3}{2}$\\ \cline{1-5}
			\end{tabular}
		\end{center}
		
\vspace{-6pt}		\noindent where $a$ is any integer such that $\max\{0,6n{-}5m\}{\leq} a{\leq} 3n{-}2m{-}1$
		when $n{+}1{<} m\leq \frac{3n-3}{2}$.				
	\end{defn}
	
It is not difficult to show that for every $x_{ij}$ values of the five families defined above, there is at least one chemical graph having $x_{ij}$ $(i,j)$-edges. This can be proved in several ways. The first approach is to use the necessary and sufficient conditions  provided in Hansen~\emph{et al.}~\cite{Hansen2017} for the existence of a simple connected graph with given $x_{ij}$ values. These conditions for chemical graphs can be written as follows:
\begin{align}
   \label{cond1}  x_{33} & \leq \tfrac{n_3(n_3-1)}{2} &&\mbox{ if } n_3 = 1, 2 \mbox{ or } 3,\\[-2pt]
   \label{cond2}  x_{22} & \leq \tfrac{n_2(n_2-1)}{2}  &&\mbox{ if }  n_2 = 1 \mbox{ or }  2,\\[-2pt]
   \label{cond3}  x_{23} & \leq n_2n_3  &&\mbox{ if }  n_2 = 1 \mbox{ or }  2 \mbox{ and }  n_3 = 1,\\[-2pt]
   \label{cond4}  x_{23} & \geq \delta(n_2) + \delta(n_3) - 1,\\[-2pt]
   \label{cond5}  x_{23} + x_{33} & \geq n_3 + \delta(n_2) - 1,\\[-2pt]
   \label{cond6}  x_{22} + x_{23} & \geq n_2 + \delta(n_3) - 1,\\[-2pt]
   \label{cond7}  x_{22} + x_{23} + x_{33} & \geq n_2 + n_3 - 1.
\end{align}
where 
$$
\delta(x) = \left\{ 
\begin{array}{ll}
1 & \mbox{if } x \ge 1,\\
0 & \mbox{otherwise.}
\end{array}
\right.
$$

Condition~\eqref{cond7} is equivalent to $m - x_{12} - x_{13} \geq n - x_{12} - x_{13} - 1$, which is equivalent to $m \geq n - 1$. 
In summary, given a pair $(n,m)$ of integers such that $n\geq 7$ and  $n-1\leq m\leq \frac{3n-3}{2}$, and given $x_{ij}$ values that satisfy conditions~\eqref{eq_n1}-\eqref{m}, we can state that there is a chemical graph of order $n$ and size $m$ with $x_{ij}$ $(i,j)$-edges if and only if conditions~\eqref{cond1}-\eqref{cond6} are satisfied. This is now illustrated with family $F_1$. 

Given $x_{ij}$ values as in Definition \ref{def_F1}, Equations~\eqref{eq_n1}, \eqref{eq_n2} and \eqref{eq_n3} give
$$\left\{ 
\begin{array}{ll}
n_1 & = \frac{3n - 2m - (n\mymod 2)}{2},\\[0pt]
n_2 & = n \mymod 2,\\[0pt]
n_3 & = \frac{2m - n - (n\mymod 2)}{2}.
\end{array}
\right.$$
Clearly, $n=n_1+n_2+n_3$ and $m=x_{12}+x_{13}+x_{22}+x_{23}+x_{33}$, which means that conditions \eqref{eq_n1}-\eqref{m} are satisfied. Let's now prove that Constraints~\eqref{cond1}-\eqref{cond6} are also satisfied. 
Note first that $n\geq 7$ implies $m\geq n-1\geq 6$. Since $3n_3=m+x_{33}$, we have $n_3\geq 2$. 
\begin{itemize}[nosep]
	\item 	If $n_3=2$ then $6=m+x_{33}\geq 6+x_{33}$ implies $x_{33} <1=\frac{n_3(n_3-1)}{2}$; if $n_3=3$, then $9=m+x_{33}\geq 6+x_{33}$ implies $x_{33} \leq 3=\frac{n_3(n_3-1)}{2}$. Hence, Constraint~\eqref{cond1} is satisfied. 
	\item Since $x_{22}=0\leq \frac{n_2(n_2-1)}{2}$ for $n_2=1$ and $2$,  Constraint~\eqref{cond2} is satisfied.
	\item As mentioned above, $n_3\geq 2$ which implies that there is no Constraint~\eqref{cond3}.
	\item 	If $n$ is even, then
	$x_{23}=n_2=\delta(n_2)=0$. Therefore,
	\begin{itemize}[nosep]
		\item $x_{23}=0\geq \delta(n_2)+\delta(n_3)-1$;
		\item Since $m\geq n-1$, we have $2x_{33}=4m-3n\geq m-3$. Hence, $m-3+x_{33}\leq 3x_{33}$ which implies $x_{23}+x_{33}=x_{33}\geq \frac{m+x_{33}-3}{3}=n_3-1=n_3+\delta(n_2)-1$;
	\item $x_{22}+x_{23}\geq n_2+\delta(n_3)-1$.
\end{itemize}
Hence, Constraints~\eqref{cond4}-\eqref{cond6} are satisfied.
	\item If $n$ is odd, then,  $x_{23}=2$ and $n_2=\delta(n_2)=1$. Therefore,
	\begin{itemize}[nosep]
		\item $x_{23}=2 > \delta(n_2)+\delta(n_3)-1$.
	\item Since $m\geq n-1$, we have $2x_{33}=4m-3n-3\geq m-6$. Hence, $m-6+x_{33}\leq 3x_{33}$ which implies $x_{23}+x_{33}=2+x_{33}\geq \frac{m+x_{33}}{3}=n_3=n_3+\delta(n_2)-1$. 
	\item $x_{22}+x_{23}=2> n_2+\delta(n_3)-1$.
\end{itemize} Hence, Constraints~\eqref{cond4}-\eqref{cond6} are satisfied.	
\end{itemize}

\vspace{6pt}Another way of proving that a chemical graph exists for given  $x_{ij}$ values is to give an explicit construction for such a graph. For family $F_1$, for an even order $n\geq 7$ and for $m\geq n$, one can for example consider the following construction (a similar one can be given for odd values of $n$ and for $m=n-1$):
\begin{itemize}[nosep]
	\item[1.] Construct a cycle on vertices $v_1, v_2, \ldots, v_{m-\frac{n}{2}}$, with edges $v_i v_{i+1}$ ($1\leq i \leq m - \frac{n}{2} - 1$) and $v_1v_{m - \frac{n}{2}}$.
	\item[2.] Add a matching with the $m-n$ edges  $v_i v_{\lceil \frac{2m-n}{4} \rceil+i}$ ($1<i\leq m-n$). Let $W$ be the set of endpoints of these edges.
	\item[3.] For each $v_i \notin W$, add a pending vertex $w_i$ adjacent to $v_i$. 
\end{itemize}
The resulting graph belongs to $F_1$. Indeed, every $v_i$ has degree 3 and every $w_i$ has degree 1. We thus have $x_{12} = x_{22} = x_{23} = 0$. Moreover, 
$x_{13} = m - \frac{n}{2} - |W| = \frac{3n-2m}{2}$ and 
$x_{33} = m - \frac{n}{2} + |W| =  \frac{4m-3n}{2}.$

In summary, it is tedious but easy to check that given $x_{ij}$ values of one of the five graph families defined above, there is at least one chemical graph with $x_{ij}$ $(i,j)$-edges. Therefore, from now on, we assume that 
this is true for the five families $F_1,\ldots,F_5$.

\section{Tools used to characterize extremal chemical graphs} 

\vspace{-0.3cm}Given a set of $x_{ij}$ values, we consider transformations which generate $x'_{ij}$ values having specific properties.

\begin{defn}
	Let $A=(a_{12},a_{13},a_{22},a_{23},a_{33})$ be a vector with integer coefficients. 
	
	\vspace{-11pt}\begin{itemize}\setlength{\itemsep}{-3pt}	
		\item Given any integer $k$, the $(A,k)$-transform of a vector $(x_{12},x_{13},x_{22},x_{23},x_{33})$ is the vector $(x'_{12},x'_{13},x'_{22},x'_{23},x'_{33})$ such that $x'_{ij}=x_{ij}+ka_{ij}$.
		\item 
	We say that $A$ is $(n,m)$-preserving if it satisfies the two following equations:
\vspace{-0.2cm}\begin{align}
\label{eq_n_aij}    \frac{3}{2} a_{12} + \frac{4}{3} a_{13} + a_{22} + \frac{5}{6} a_{23} + \frac{2}{3} a_{33} & = 0;\\[-3pt]
\label{eq_m_aij}   a_{12}+a_{13}+a_{22}+a_{23}+a_{33} & = 0.
\end{align}
\end{itemize}
\end{defn}

\vspace{-0.4cm}The idea behind these definitions is that if $(x'_{12},x'_{13},x'_{22},x'_{23},x'_{33})$ is the $(A,k)$-transform of $(x_{12},x_{13},x_{22},x_{23},x_{33})$ and if $A$ is $(n,m)$-preserving, then 
\vspace{-0.2cm}$$\frac{3}{2} x_{12} + \frac{4}{3} x_{13} + x_{22} + \frac{5}{6} x_{23} + \frac{2}{3} x_{33}=\frac{3}{2} x'_{12} + \frac{4}{3} x'_{13} + x'_{22} + \frac{5}{6} x'_{23} + \frac{2}{3} x'_{33} \mbox{ \quad\quad and}\vspace{-0.3cm}$$
$$x_{12}+x_{13}+x_{22}+x_{23}+x_{33}=x'_{12}+x'_{13}+x'_{22}+x'_{23}+x'_{33}. 
$$

Hence, the values of $n$ and $m$ derived from Equations~\eqref{n} and \eqref{m} are the same, whether calculated using $x_{ij}$ or $x'_{ij}$ values.

Let $G$ be a chemical graph of order $n$ and size $m$ that maximizes the value of a topological index $f$ over all chemical graphs of same order and same size as $G$. We now study the impact on the $x_{ij}$ values of $G$ if some of the four following values are strictly positive: 
    \begin{align*}
        V_1&=c_{13}-c_{22}+\min_{i=1,2,3}{({c_{3i}}-c_{2i})}+\min_{j=2,3}{({c_{3j}}-c_{2j})}, \\[-3pt]
        V_2&=c_{13}-c_{12}+\min_{i=2,3}{({c_{2i}}-c_{3i})}, \\[-3pt]
        V_3&=c_{22}-c_{13}+\min_{i=1,2,3}{({c_{2i}}-c_{3i})}+\min_{j=2,3}{({c_{2j}}-c_{3j})}, \\[-3pt]
       V_4&={2c}_{22}-c_{12}-c_{23}+2\min_{i=1,2,3}{({c_{2i}}-c_{3i})}.
    \end{align*}
\begin{lem}\label{lem:V1}
	 Let $f$ be a degree-based topological index  such that $V_1>0$. A chemical graph $G$ that maximizes $f$ over all chemical graphs of the same order and size as $G$ has no (2,2)-edge.
	\end{lem}
	\begin{proof}
		Assume $G$ contains an edge $uv$ with both endpoints of degree 2. 
		\begin{itemize}[nosep]
			\item If $u$ and $v$ have no common neighbor, then let $x$ be the second neighbor of $u$ and let $y$ be the second neighbor of $v$. At least one of $x,y$, say $y$, has degree at least two, else $G$ has order 4. We can then obtain a chemical graph $G'$ by replacing $ux$ by $vx$. Let $i\in\{1,2,3\}$ be the degree of $x$ and $j\in\{2,3\}$ the degree of $y$. The graph $G'$ contradicts the maximality of $G$ since 
			\vspace{-8pt}$$f(G')=f(G)+c_{13}-c_{22}+(c_{3i}-c_{2i})+(c_{3j}-c_{2j})\geq f(G)+V_1>f(G).$$
			\item If $u$ and $v$ have a common neighbor $w$, then $w$ has degree 3, else $G$ has order $n=3$. Also, the third neighbor $x$ of $w$ has degree at least two, else $G$ has order $n=4$.
			\begin{itemize}[nosep]
				\item If $x$ has degree 2, then let $y$ be its second neighbor and let $i\in\{1,2,3\}$ be the degree of $y$. We can obtain a chemical graph $G'$ by replacing $uv$ by $vx$. Then $G'$ contradicts the maximality of $G$ since \begin{center}$f(G')=f(G)+c_{13}-c_{22}+            (c_{3i}-c_{2i})+(c_{33}-c_{23})\geq f(G)+V_1>f(G).$\end{center}
			\item If $x$ has degree 3, then let $y$ and $z$ be the two other neighbors of $x$. We can obtain a chemical graph $G'$ by replacing $xy$ and $xz$ by $uy$ and $vz$.   Then $G'$ contradicts the maximality of $G$ since
\begin{center}\hspace{1cm}$f(G')=f(G)+c_{13}-c_{22}+ 2(c_{33}-c_{23})  \geq f(G)+V_1>f(G).\hfill\qedhere$\end{center}
			\end{itemize}
		\end{itemize}
	\end{proof}

\begin{lem}\label{lem:V2}
	Let $f$ be a degree-based topological index  such that $V_1>0$ and $V_2>0$. A chemical graph $G$ that maximizes $f$ over all chemical graphs of the same order and size as $G$ has no (1,2)-edge and no (2,2)-edge.
\end{lem}
\begin{proof}
	We already know from Lemma \ref{lem:V1} that $G$ has no (2,2)-edge. Let $uv$ be an edge in $G$ with $u$ of degree 1 and $v$ of degree $2$. Let $w$ be the second neighbor of $v$. Note that $w$ has degree 3, else $G$ has order $n=3$. Let $x$ and $y$ be the two other neighbors of $w$. At least one of them, say $x$ has degree $i\geq 2$, else $G$ has order $n=5$. We can obtain a chemical graph $G'$ by replacing $uv$ and $wx$ by $uw$ and $vx$. Then $G'$ contradicts the maximality of $G$ since
\begin{center}\hspace{1cm}$f(G')= f(G)+c_{13}-c_{12}+(c_{2i}-c_{3i})\geq f(G)+V_2>f(G).\hfill\qedhere$\end{center}
\end{proof}

\begin{lem}\label{lem:V3}
	Let $f$ be a degree-based topological index  such that $V_3>0$. A chemical graph $G$ that maximizes $f$ over all chemical graphs of the same order and size as $G$ has no (1,3)-edge.
\end{lem}
\begin{proof}
	Let $uv$ be an edge with $u$ of degree 1 and $v$ of degree $3$, and let $x$ and $y$ be the two other neighbors of $v$. At least one of $x$ and $y$, say $y$ has degree at least two, else $G$ has order $n=4$. We can then obtain a chemical graph $G'$ by replacing $vy$ by $uy$. Let $i$ be the degree of $x$ and $j$ the degree of $y$.  Then $G'$ contradicts the maximality of $G$ since
	\begin{center}\hspace{1cm}$f(G')=f(G)+c_{22}-c_{13}+(c_{2i}-c_{3i})+(c_{2j}-c_{3j})\geq f(G)+V_3>f(G).\hfill\qedhere$\end{center}
	\end{proof}

\begin{lem}\label{lem:V4}
	Let $f$ be a degree-based topological index  such that $V_4>0$. If a chemical graph $G$  maximizes $f$ over all chemical graphs of the same order and size as $G$ then either $m=n-1$ and $G$ is $\Path{n}$ or $m\geq n$ and $G$ has no (1,2)-edge.
\end{lem}
\begin{proof}
	Let $uv$ be an edge with $u$ of degree 1 and $v$ of degree $2$. If $G$ is not $\Path{n}$, then there is a vertex $w$ of degree 3 in $G$ such that $v$ and $w$ are linked by a chain in which all internal vertices have degree 2 in $G$. Let $x$ and $y$ be the two neighbors of $w$ that are not on the chain. We can obtain a chemical graph $G'$ by replacing $wx$ by $ux$. Let $i$ be the degree of $x$ and $j$ the degree of $y$.  Then $G'$ contradicts the maximality of $G$ since
	\begin{center}\hspace{1cm}$f(G')=f(G)+
	2c_{22}-c_{12}-c_{23}+(c_{2i}-c_{3i})+(c_{2j}-c_{3j})\geq f(G)+V_4>f(G).\hfill\qedhere$\end{center}
\end{proof}

\section{Characterization of extremal graphs }\label{secFamilies}

We first characterize the extremal graphs of 29 of the 33 degree-based topological indices of Table \ref{tab_inv}. The proofs involve the following values:
    \begin{align*}
    V_5&=c_{13}-4c_{23}+{3c}_{33}, \\[-2pt]
    V_6&={c}_{22}-2c_{23}+c_{33}, \\[-2pt]
    V_7&=c_{12}-c_{13}-c_{23}+c_{33}, \\[-2pt]
    V_8&=-2c_{12}+3c_{13}-2c_{23}+c_{33}.
    \end{align*}
\subsection{Five graph families for 29 topological indices}
We first  show that the chemical graphs in $F_1$ maximize all degree-based topological indices  such that $V_1, V_2$ and $V_5$ are strictly positive.
\begin{thm}\label{thm_F1}
	Let $f$ be a degree-based topological index such that $V_1>0$, $V_2>0$ and $V_5>0$. A chemical graph $G$ maximizes $f$ over all chemical graphs of the same order and size as $G$ if and only if $G\in F_1$.
\end{thm}

\begin{proof}
 Let $G$ be a chemical graph of order $n$, size $m$ and with $x_{ij}$ $(i,j)$-edges. Assume that it maximizes $f$ over all chemical graphs of order $n$ and size  $m$. As shown in Lemma \ref{lem:V2}, $V_1>0$ and $V_2>0$ imply $x_{12}=x_{22}=0$. Hence, $2n_2 = x_{23}$, which means that $x_{23}$ is even.
 
 If $x_{23}\leq 2$ then $n_2 = n \mymod 2$ (since the number $n_1+n_3$ of odd degree vertices is even), which implies $x_{23}=2(n \mymod 2)$. Equations \eqref{n} and \eqref{m} then give $x_{13} = \frac{3n - 2m - (n \mod 2)}{2}$ and 
$x_{33} = \frac{4m - 3n - 3(n \mod 2)}{2}$, which means that $G$ belongs to $F_1$.
 
 If $x_{23} \ge 4$, then consider the vector $A=(0,1,0,-4,3)$ associated with $V_5$ and let $(x'_{12},x'_{13},x'_{22},x'_{23},x'_{33})$ be the $(A,\lfloor \frac{x_{23}}{4}\rfloor)$-transform of $(x_{12},x_{13},x_{22},x_{23},x_{33})$. We then have $x'_{12}=x'_{22}=0$ and $x'_{23}\leq 2$. Since $A$ is $(n,m)$-preserving, we conclude as above that 
 $x'_{13}=\frac{3n - 2m - (n \mod 2)}{2}$ and $x'_{33} = \frac{4m - 3n - 3(n \mod 2)}{2}$. Let $G'$ be a graph in $F_1$ having exactly $x'_{ij}$ $(i,j)$-edges. The maximality of $G$ is contradicted by $G'$ since  \begin{center}\hspace{1cm}$f(G')=f(x'_{12},x'_{13},x'_{22},x'_{23},x'_{33})=f(G)+V_5\lfloor \frac{x_{23}}{4}\rfloor>f(G).\hfill\qedhere$\end{center}
\end{proof}

\begin{thm}\label{thm_F2}
	Let $f$ be a degree-based topological index such that  $V_3>0$, $V_4>0$ and $V_6>0$. A chemical graph $G$ maximizes $f$ over all chemical graphs of the same order and size as $G$ if and only if $G\in F_2$.
\end{thm}

\begin{proof}
Let $G$ be a chemical graph of order $n$, size $m$ and with $x_{ij}$ $(i,j)$-edges. Assume that it maximizes $f$ over all chemical graphs of order $n$ and size  $m$.  As shown by Lemmas \ref{lem:V3} and \ref{lem:V4}, $V_3>0$ and $V_4>0$ imply $G=\Path{n}$ or $n_1=0$. Hence, if $m=n-1$, then $G=\Path{n}$ (since trees have vertices of degree 1), which means that $x_{12}=2$, $x_{22}=m-2$ and $G\in F_2$. So assume $n\leq m<\frac{3n-3}{2}$, which implies that $n_1=0$ and $x_{23}$ is even.
	\begin{itemize}[nosep]
		\item If $m=n$, Equations~\eqref{n} and \eqref{m} give $x_{23}=x_{33}=0$, which implies $x_{22}=m$ and $G\in F_2$.
		\item If $m=n+1$, Equations~\eqref{n} and \eqref{m} give $x_{23}+2x_{33}=6$, which implies $n_3=2$ and $x_{33}\leq 1$. Hence, there are only two possibilities:
		\begin{itemize}[nosep]
			\item $x_{22}=m-6$, $x_{23}=6$,$x_{33}=0$, or
			\item $x_{22}=m-5$, $x_{23}=4$,$x_{33}=1$.
		\end{itemize}	
		Since $c_{22}{-}2c_{23}{+}c_{33}{=}V_6>0$, the second solution has a larger value, which implies $G{\in} F_2$.
		\item If $n+1< m<\frac{3n-3}{2}$ and $x_{23}=2$, then Equations~\eqref{n} and \eqref{m} give $x_{22}=3n-2m-1$ and $x_{33}=3m-3n-1$, which implies $G\in F_2$. So assume $x_{23}\geq 4$, consider the vector $A=(0,0,1,-2,1)$ associated with $V_6$ and let $(x'_{12},x'_{13},x'_{22},x'_{23},x'_{33})$ be the $(A, \frac{x_{23}-2}{2})$-transform of $(x_{12},x_{13},x_{22},x_{23},x_{33})$. Hence, $x'_{12}=x'_{13}=0$ and $x'_{23}=2$. Since $A$ is $(n,m)$-preserving, we conclude as above that $x'_{22}=3n-2m-1$ and $x'_{33}=3m-3n-1$. Consider any chemical graph $G'$ in $F_2$ having exactly $x'_{ij}$ $(i,j)$-edges. The maximality of $G$ is contradicted by $G'$ since  	\vspace{-5pt}\begin{center}\hspace{1cm}$f(G')=f(x'_{12},x'_{13},x'_{22},x'_{23},x'_{33})=f(G)+V_6(\frac{x_{23}-2}{2})>f(G).\hfill\qedhere$\end{center}
\end{itemize}
\end{proof}

\begin{cor}\label{cor_F12}
	$F_1\cup F_2$ is the set of extremal chemical graphs for the 13 degree-based topological indices ABSC, AG, AG-GA, Extended, GA, GouravaSC, Harmonic, Randi\'c,  Sombor, rSombor,  SumConn, rSumConn and lnZagreb1. 
\end{cor}
\begin{proof}
	It is easy to check that 
	\begin{itemize} [nosep]
		\item $V_1,V_2$ and $V_5$ are strictly positive for ABSC, AG, AG-GA,  Extended, rSumConn, Sombor, rSombor, lnZagreb1,  \textoverline{GA}, \textoverline{GouravaSC}, \textoverline{Harmonic},  \textoverline{Randi\'c} and \textoverline{SumConn}, which implies that $F_1$ is the set of chemical graphs which maximize ABSC, AG, AG-GA,  Extended, rSumConn, Sombor, rSombor, lnZagreb1 and minimize GA,  GouravaSC, Harmonic, Randi\'c and SumConn.
		\item $V_3,V_4$ and $V_6$ are strictly positive for 		\textoverline{ABSC}, \textoverline{AG}, \textoverline{AG-GA}, \textoverline{Extended}, \textoverline{Sombor}, \textoverline{rSombor}, \textoverline{rSumConn}, \textoverline{lnZagreb1}, GA, GouravaSC, Harmonic,Randi\'c and SumConn, which implies that $F_2$ is the set of chemical graphs which minimize ABSC, AG,AG-GA, Extended, Sombor, rSombor, rSumConn, lnZagreb1 and maximize GA, GouravaSC,Harmonic,Randi\'c and SumConn.\qedhere
		\end{itemize}
\end{proof}

We now characterize the degree-based topological indices $f$ for which the chemical graphs in $F_3$ and $F_4$ maximize $f$.

\begin{thm}\label{thm_F3}
	Let $f$ be a degree-based topological index such that  $V_1{>}0, V_6{>}0, V_7{>}0$ and $V_8{>}0$. A chemical graph $G$ maximizes $f$ over all graphs of the same order and size as $G$ if and only if $G\in F_3$.
\end{thm}
\begin{proof}
Let $G$ be a chemical graph of order $n$, size $m$ and with $x_{ij}$ $(i,j)$-edges. Assume that it maximizes $f$ over all chemical graphs of order $n$ and size  $m$.   
	As shown by Lemma \ref{lem:V1}, $V_1>0$ implies $x_{22}=0$. Let $W_{33}$ be the set of vertices of degree 2 in $G$ with two neighbors of degree 3, and let $W_{13}$ be the set of vertices of degree 2 in $G$ with one neighbor of degree 1 and the other of degree 3. Since $n>3$, we have $n_2=|W_{33}| +|W_{13}|$.
	\begin{itemize}[nosep]
		\item If $n_2=0$, then $x_{12}=x_{22}=x_{23}=0$ and Equations~\eqref{n} and \eqref{m} give $x_{13}=\frac{3n-2m}{2}$ and $x_{33}=\frac{4m-3n}{2}$, which implies that $G$ belongs to $F_3$.
		\item If $n_2=1$, then let $v$ be the vertex of degree 2 in $G$.
        \begin{itemize}[nosep]
        \item if $v\in W_{13}$ then $x_{12}=x_{23}=1,x_{22}=0$  and Equations~\eqref{n} and \eqref{m} give $x_{13}=\frac{3n-2m-3}{2}$ and $x_{33}=\frac{4m-3n-1}{2}$;
        \item if $v\in W_{33}$ then $x_{12}=x_{22}=0,x_{23}=2$ and Equations~\eqref{n} and \eqref{m} give $x_{13}=\frac{3n-2m-1}{2}$ and $x_{33}=\frac{4m-3n-3}{2}$.
        \end{itemize}
        Since $c_{12}-c_{13}-c_{23}+c_{33}=V_7>0$, the first case has a larger value $f(G)$, which implies that $G$ belongs to $F_3$.
		\item If $n_2>1$, then the two neighbors of each vertex in $W_{33}$ are adjacent. Indeed, if a vertex $v\in W_{33}$ has two non-adjacent neighbors $u_1$ and $u_2$, then consider any other vertex $w$ of degree 2  and let $u_3$ be one of its neighbors: by replacing $vu_1,vu_2,wu_3$ by $vw,vu_3$ and $u_1u_2$, we get a chemical graph $G'$ which contradicts the maximality of $G$ since        \vspace{-4pt}$$f(G')=f(G)+c_{22}-2c_{23}+c_{33}=f(G)+V_6>f(G).$$
        
        \vspace{-4pt}Let us now show that $|W_{33}|\leq 1$. Assume by contradiction that $W_{33}$ contains at least two vertices $v_1$ and $v_2$. Since $n>4$,  there are two non-adjacent vertices $u_1$ and $u_2$ such that $u_1$ is adjacent to $v_1$ but not to $v_2$, while $u_2$ is adjacent to $v_2$ but not to $v_1$. Let $w_1$ be the second neighbor of $v_1$, and let $G'$ be the chemical graph obtained from $G$ by replacing $v_1w_1$ and $v_2u_2$ by $v_1u_2$ and $v_2w_1$. Then $G'$ has $n_2>1$ vertices of degree 2 and one of them, namely $v_1$, has two non-adjacent neighbors $u_1,u_2$ of degree 3. We have shown above that this implies that $G'$ does not maximize $f$ while $f(G')=f(G)$, a contradiction.
        
        Hence, $|W_{33}|\leq 1$, which implies $|W_{13}|\geq 1$. So let $v$ be a vertex in $W_{13}$, let $u$ be another vertex of degree 2, let $w$ be the neighbor of $v$ of degree 1, and let $G'$ be the chemical graph obtained from $G$ by replacing $vw$ by $uw$:
        \begin{itemize}[nosep]
        \item if $u\in W_{13}$ then $f(G')=f(G)-2c_{12}+3c_{13}-2c_{23}+c_{33}=f(G)+V_8>f(G)$;
        \item if $u\in W_{33}$ then $f(G')=f(G)-c_{12}+2c_{13}-3c_{23}+2c_{33}=f(G)+V_7+V_8>f(G)$.
        \end{itemize}
        In both cases, $G'$ contradicts the maximality of $G$.\qedhere
	\end{itemize} 
\end{proof}

\begin{thm}\label{thm_F4}
	Let $f$ be a degree-based topological index such that  $V_3{>}0$, $V_4{>}0$ and $V_6{<}0$. A chemical graph $G$ maximizes $f$ over all graphs of the same order and size as $G$ if and only if $G\in F_4$.
\end{thm}
\begin{proof}
Let $G$ be a chemical graph of order $n$, size $m$ and with $x_{ij}$ $(i,j)$-edges. Assume that it maximizes $f$ over all chemical graphs of order $n$ and size  $m$.    As shown by Lemmas \ref{lem:V3} and \ref{lem:V4}, $V_3>0$ and $V_4>0$ imply $G=\Path{n}$ or $n_1=0$. Hence, if $m=n-1$, then $G=\Path{n}$, which means that $x_{12}=2$, $x_{22}=m-2$ and $G\in F_4$. So assume $m>n-1$, which implies $n_1=0$. Equations \eqref{n} and \eqref{m} give $x_{23}=6n-4m-2x_{22}$ and $x_{33}=5m-6n+x_{22}$. Hence, $x_{22}\geq \max\{0, 6n-5m\}$.
    \begin{itemize}[nosep]
		\item If $x_{22}=0$, then $x_{23}=6n-4m$, $x_{33}=5m-6n$ and $G\in F_4$.
        \item If $x_{22}=6n-5m>0$, then $x_{23}=6m-6n$, $x_{33}=0$ and $G\in F_4$.
        \item If $x_{22} > 0$ and $x_{22} \neq 6n - 5m$, then $x_{33}>0$. Consider the vector $A=(0,0,-1,2,-1)$ associated with $-V_6$ and let $(x'_{12},x'_{13},x'_{22},x'_{23},x'_{33})$ be the $(A, \min\{x_{22},x_{22}-6n+5m\})$-transform of $(x_{12},x_{13},x_{22},x_{23},x_{33})$. Hence, $x'_{22}=\max\{0, 6n-5m\}$ and $x'_{12}=x'_{13}=0$. Since $A$ is $(n,m)$-preserving we conclude as above that there is a chemical graph $G'$ in $F_4$ having exactly $x'_{ij}$ $(i,j)$-edges. Then $G'$ contradicts the maximality of $G$ since  \vspace{-4pt}\begin{center}\hspace{1cm}$f(G')=f(x'_{12},x'_{13},x'_{22},x'_{23},x'_{33})=f(G)-V_6\min\{x_{22},x_{22}-6n+5m\}>f(G).\hfill\qedhere$\end{center}
    \end{itemize}
\end{proof}

\begin{cor}\label{cor_F34}
	$F_3\cup F_4$ is the set of extremal chemical graphs for the 10 degree-based topological indices Gourava1, Gourava2, hGourava1, hGourava2, GouravaPC, InvSumDeg, rRandi\'c, Zagreb2, hZagreb1, hZagreb2.
\end{cor}
\begin{proof}
It is easy to check that 
	\begin{itemize} [nosep]
		\item $V_1{>}0$, $V_6{>}0$, $V_7{>}0$ and $V_8{>}0$ for  Gourava1, Gourava2, hGourava1, hGourava2, GouravaPC, InvSumDeg, rRandi\'c, Zagreb2, hZagreb1, hZagreb2, which implies that $F_3$ is the set of chemical graphs which maximize these topological indices.
		\item $V_3{>}0,V_4{>}0$ and $V_6{<}0$ for 
		\textoverline{Gourava1}, 
\textoverline{Gourava2}, 
\textoverline{hGourava1}, \textoverline{hGourava2},
\textoverline{ GouravaPC}, 
\textoverline{InvSumDeg}, 
\textoverline{rRandi\'c}, 
		\textoverline{Zagreb2}, 
		\textoverline{hZagreb1}, \textoverline{hZagreb2}, 				
which implies that $F_4$ is the set of chemical graphs which minimize the 10 topological indices.\qedhere
	\end{itemize}
\end{proof}

Note that $F_1\cup F_3$ is the set of chemical graphs with $x_{ij}$ $(i,j)$-edges such that 
		\begin{center}
			\setlength{\extrarowheight}{2pt}
			\begin{tabular}{| c | c | c | c |c | l}\cline{1-5} 
				$x_{12}$ & $x_{13}$ & $x_{22}$& $x_{23}$& $x_{33}$&\\ \cline{1-5}
				0 & $\frac{3n - 2m }{2}$ & 0 & 0 &$\frac{4m - 3n}{2}$& if $n$ if even \\ \cline{1-5}
				1 & $\frac{3n - 2m-3 }{2}$ & 0 & 1 &$\frac{4m - 3n-1}{2}$& \multirow{2}{*}{if $n$ is odd}\\ 				
				0 & $\frac{3n - 2m-1 }{2}$ & 0 & 2 &$\frac{4m - 3n-3}{2}$& \\ \cline{1-5}
			\end{tabular}
		\end{center}

\begin{thm}\label{thm_F13}
	Let $f$ be a degree-based topological index such that  $V_1{>}0$, $V_5{>}0$ and $V_7{=}0$. A chemical graph $G$ maximizes $f$ over all graphs of the same order and size as $G$ if and only if $G\in F_1\cup F_3$.
\end{thm}
\begin{proof}
Let $G$ be a chemical graph of order $n$, size $m$ and with $x_{ij}$ $(i,j)$-edges. Assume that it maximizes $f$ over all chemical graphs of order $n$ and size  $m$.    Note  that the two possibilities for the $x_{ij}$ values when $n$ is odd give the same value $f(G)$ since $c_{12}-c_{13}-c_{23}+c_{33}=V_7=0$.
As shown by Lemma \ref{lem:V1}, $V_1>0$ implies $x_{22}=0$. Moreover, $n>3$ implies $x_{12}\leq x_{23}$, and $x_{12}$ and $x_{23}$ have the same parity.
	\begin{itemize}[nosep]
		\item[1.] If $x_{12}=0$ then, as shown in Theorem \ref{thm_F1}, $G\in F_1$, else there is a graph $G'\in F_1$ so that $f(G')>f(G)$. 
		\item[2.] If $x_{12}=1$ then
		\begin{itemize}[nosep]
		\item if $x_{23}=1$ then Equations~\eqref{n} and \eqref{m} give $x_{13}=\frac{3n-2m-3}{2}$ and $x_{33}=\frac{4m-3n-1}{2}$, which implies that $G$ belongs to $F_3$.
		\item if $x_{23}\geq 3$, then consider the vector $A=(-1,2,0,-3,2)$ associated with $V_5{-}V_7$ and let $(x'_{12},x'_{13},x'_{22},x'_{23},x'_{33})$ be the $(A,1)$-transform of $(x_{12},x_{13},x_{22},x_{23},x_{33})$. Note that $x'_{12}=x'_{22}=0$ and $A$ is $(n,m)$-preserving. Hence, we have shown in case 1. that there is a graph $G'\in F_1$ which contradicts the maximality of $G$ since 
		\begin{center}$f(G')\geq f(x'_{12},x'_{13},x'_{22},x'_{23},x'_{33})=f(G)+V_5-V_7=f(G)+V_5>f(G).$\end{center}
		\end{itemize}
		\item[3.] If $x_{12}\geq 2$ then $x_{23}\geq x_{12}\geq 2$. Consider the vector $A=(-2,3,0,-2,1)$ associated with $V_5-2V_7$ and let $(x'_{12},x'_{13},x'_{22},x'_{23},x'_{33})$ be the $(A,\lfloor \frac{x_{12}}{2} \rfloor)$-transform of $(x_{12},x_{13},x_{22},x_{23},x_{33})$. Since $x'_{12}=x'_{22}=0$ and $A$ is $(n,m)$-preserving, we have shown in case 1. that there is a graph $G'\in F_1$ which contradicts the maximality of $G$ since 
		\vspace{-5pt}\begin{center}\hspace{0.5cm}$f(G'){\geq} f(x'_{12},x'_{13},x'_{22},x'_{23},x'_{33})=f(G)+\lfloor \frac{x_{12}}{2} \rfloor (V_5{-}2V_7)=f(G)+\lfloor \frac{x_{12}}{2} \rfloor V_5{>}f(G). \hfill\qedhere$\end{center}
	\end{itemize}
\end{proof}

\begin{thm}\label{thm_F5}
	Let $f$ be a degree-based topological index such that  $V_3{>}0$, $V_4{>}0$ and $V_6{=}0$. A chemical graph $G$ maximizes $f$ over all graphs of the same order and size as $G$ if and only if $G\in F_5$.
\end{thm}
\begin{proof}
Let $G$ be a chemical graph of order $n$, size $m$ and with $x_{ij}$ $(i,j)$-edges. Assume that it maximizes $f$ over all chemical graphs of order $n$ and size  $m$.    Note first that the two possibilities for the $x_{ij}$ values when $m=n+1$ have the same value $f(G)$ since $c_{22}-2c_{23}+c_{33}=V_6=0$. For the same reason, all solutions for
$n+1< m \leq \frac{3n-3}{2}$ have the same value.

 As shown by Lemmas \ref{lem:V3} and \ref{lem:V4}, $V_3>0$ and $V_4>0$ imply $G=\Path{n}$ or $n_1=0$. Hence, if $m=n-1$, then $G=\Path{n}$, which means that $x_{12}=2$, $x_{22}=m-2$ and $G\in F_5$. 
	So assume $n\leq m<\frac{3n-3}{2}$, which implies that $n_1=0$ and $x_{23}$ is even.
	\begin{itemize}[nosep]
		\item If $m=n$, Equations~\eqref{n} and \eqref{m} give $x_{23}=x_{33}=0$ and $x_{22}=m$, which implies $G\in F_5$.
		\item If $m=n+1$, Equations~\eqref{n} and \eqref{m} give $x_{23}+2x_{33}=6$, which implies $n_3=2$ and $x_{33}\leq 1$. Hence, there are only two possibilities:
		\begin{itemize}[nosep]
			\item $x_{22}=m-6$, $x_{23}=6$,$x_{33}=0$, or
			\item $x_{22}=m-5$, $x_{23}=4$,$x_{33}=1$,
		\end{itemize}	
		and both imply $G\in F_5$.
		\item If $n+1{<} m{\leq}\frac{3n-3}{2}$, then $x_{23}\geq 2$ and Equations~\eqref{n} and \eqref{m} give $x_{22}=a$, $x_{23}=6n-4m-2a$ and $x_{33}=5m-6n+a$. Since $x_{33}\geq 0$ and $x_{23}\geq 2$ , we have $ \max\{0,6n-5m\}\leq a\leq 3n-2m-1$, which implies $G\in F_5$.\qedhere
	\end{itemize}
\end{proof}

\begin{cor}\label{cor_F135}
	$F_1\cup F_3\cup F_5$ is the set of extremal chemical graphs for the topological indices Forgotten, InvDeg, Zagreb1, lnZagreb2, lnZagreb3 and mZagreb.
\end{cor}
\begin{proof}
	It is easy to check that 
	\begin{itemize} [nosep]
		\item $V_1{>}0$, $V_5{>}0$ and $V_7{=}0$ for Forgotten, InvDeg, Zagreb1,  \textoverline{lnZagreb2}, lnZagreb3 and mZagreb which implies that $F_1\cup F_3$ is the set of chemical graphs which maximize Forgotten, InvDeg, Zagreb1 and mZagreb and minimize lnZagreb2.
		\item $V_3{>}0,V_4{>}0$ and $V_6{=}0$ for \textoverline{Forgotten}, \textoverline{InvDeg}, \textoverline{Zagreb1}, lnZagreb2, \textoverline{lnZagreb3} and \textoverline{mZagreb},  which implies that $F_5$ is the set of chemical graphs which minimize Forgotten, InvDeg, Zagreb1, mZagreb and maximize lnZagreb2.\qedhere
	\end{itemize}
\end{proof}

\subsection{Additional families of extremal chemical graphs}

As proved in the previous section, the five families $F_1,\ldots,F_5$ are sufficient to characterize all extremal graphs of 29 topological indices. However, some degree-based topological  indices have extremal chemical graphs that do not belong to any of the five families. We give here four examples. More precisely, we characterize the extremal chemical graphs of the  topological indices ABC, Albertson, Sigma, and aZagreb. For this purpose, we define new families of chemical graphs characterized by $x_{ij}$ values.
Here again, as explained in Section \ref{sec_4families}, it is easy to check that given $x_{ij}$ values of one of the graph families, there is at least one chemical graph with $x_{ij}$ $(i,j)$-edges.

\begin{defn}\label{familyF6}
		$F_6$ is the set of chemical graphs with the following numbers $x_{ij}$ of $(i,j)$-edges:

\vspace{-10pt}		\begin{center}
			\setlength{\extrarowheight}{0pt}			\begin{tabular}{| c | c | c | c |c | l}\cline{1-5} 
				$x_{12}$ & $x_{13}$ & $x_{22}$& $x_{23}$& $x_{33}$&\\ \cline{1-5}
			0 & $a$ & 0 & $6n-4m-4a$ &$5m-6n+3a$& \\ \cline{1-5}				
			\end{tabular}
		\end{center}
\vspace{-7pt}\noindent where $a$ is any integer such that $\max\{0,\lceil \frac{6n-5m}{3}\rceil\}\leq a\leq \lfloor \frac{3n-2m}{2}\rfloor$.
\end{defn}

\begin{thm}\label{thm_F6}
Let $f$ be a degree-based topological index such that $V_1{>}0$, $V_2{>}0$ and $V_5{=}0$. A chemical graph $G$ maximizes $f$ over all graphs of the same order and size as $G$ if and only if $G\in F_6$.
\end{thm}
\begin{proof}
Let $G$ be a chemical graph of order $n$, size $m$ and with $x_{ij}$ $(i,j)$-edges. Assume that it maximizes $f$ over all chemical graphs of order $n$ and size  $m$.     As shown in Lemma \ref{lem:V2}, $V_1>0$ and $V_2>0$ imply $x_{12}=x_{22}=0$. Hence, it follows from Equations~\eqref{n} and \eqref{m} that $x_{13}=a$, $x_{23}=6n-4m-4a$ and $x_{33}=5m-6n+3a$. Since $x_{23}\geq 0$ and $x_{33}\geq 0$, we have 
	$\max\{0,\lceil \frac{6n-5m}{3}\rceil\}\leq a\leq \lfloor \frac{3n-2m}{2}\rfloor$. All possible solutions with the various values of $a$ have the same value since $c_{13}-4c_{23}+3c_{33}=V_5=0$.
	\end{proof}

\begin{defn}\label{familyF7}
		$F_7$ is the set of chemical graphs with the following numbers $x_{ij}$ of $(i,j)$-edges:
\begin{center}
			\setlength{\extrarowheight}{0pt}				\begin{tabular}{c| c | c | c | c |c | l}\cline{2-6} 
	&	$x_{12}$ & $x_{13}$ & $x_{22}$& $x_{23}$& $x_{33}$&\\ \cline{2-6}
		&2 & 0 & $m-2$ & 0 &0& if $m=n-1$ \\ \cline{2-6}
	&	0 & 0 & $m$ & 0 &0& if $m=n$ \\ \cline{2-6}
	&	0 & 0 & $m-5$ & 4 &1& \multirow{3}{*}{if $m=n+1$}\\
	&	1 & 0 & $m-7$ & 3 &3& \\ 
	$(n\geq 8)$&	2 & 0 & $m-9$ & 2 &5& \\ \cline{2-6}
	&	0 &0 &$3n-2m-1$ & 2 &$3m-3n-1$& \multirow{2}{*}{if $n+1<m\leq\frac{3n-3}{2}$}\\ 
	&	1 &0 &$3n-2m-3$ & 1 &$3m-3n+1$&\\ \cline{2-6}
		\end{tabular}
\end{center}
\end{defn}

\begin{thm}\label{thm_F7}
	Let $f$ be a degree-based topological index such that $V_3{>}0$, $V_6{>}0$ and $V_5{+}V_7{=}2V_6$. A chemical graph $G$ maximizes $f$ over all graphs of the same order and size as $G$ if and only if $G\in F_7$.
\end{thm}
\begin{proof}
Let $G$ be a chemical graph of order $n$, size $m$ and with $x_{ij}$ $(i,j)$-edges. Assume that it maximizes $f$ over all chemical graphs of order $n$ and size  $m$.    
Note that the three possible cases for $m=n+1$ have the same value since $c_{12}-2c_{22}-c_{23}+2c_{33}=V_5+V_7-2V_6=0$. For the same reason, the two possibilities for $n+1<m<\frac{3n-3}{3}$ have the same value.

As shown by Lemma \ref{lem:V3}, $V_3>0$ implies $x_{13}=0$. 
In what follows, for two integer $a$ and $b$ such that $(a,b)\neq (x_{12},x_{23})$, we say that $G$ is $(a,b)$-dominated if $a\leq x_{12}$,  $b{-}a\leq x_{23}-x_{12}$, and there is a chemical graph of order $n$ and size $m$ which has $a$ $(1,2)$-edges, $b$ $(2,3)$-edges, and no $(1,3)$-edge. In such a case
consider the $(n,m)$-preserving vector $A=(0,0,1,-2,1)$ associated with $V_6$ and let $(x'_{12},x'_{13},x'_{22},x'_{23},x'_{33})$ be the $(A,\frac{x_{23}-x_{12}+a-b}{2})$-transform of $(x_{12},x_{13},x_{22},x_{23},x_{33})$. We thus have $x_{12}=x'_{12}$ and $x'_{23}=x_{12}+b-a$. Let $A'=(-1,0,3,-1,-1)$ be the $(n,m)$-preserving vector associated with $3V_6-V_5-V_7$ and let $(x''_{12},x''_{13},x''_{22},x''_{23},x''_{33})$ be the $(A',x_{12}-a)$-transform of $(x'_{12},x'_{13},x'_{22},x'_{23},x'_{33})$. We now have $x''_{12}=a$ and $x''_{23}=b$. Let $G'$ be a graph with $x''_{ij}$ $(i,j)$-edges. Note that if $x_{12}=a$, then $x_{23}-x_{12}+a-b>0$. Hence, $G'$ contradicts the maximality of $G$ since $3V_6-V_5-V_7=V_6>0$ and 
\vspace{-0.2cm}\begin{eqnarray*}
	f(G')&=&f(x''_{12},x''_{13},x''_{22},x''_{23},x''_{33})=f(x'_{12},x'_{13},x'_{22},x'_{23},x'_{33})+(x_{12}-a)V_6\\[-1
	pt]
	&=&f(G)+(\frac{x_{23}-x_{12}+a-b}{2})V_6+(x_{12}-a)V_6>f(G),
\end{eqnarray*}	

\vspace{-0.3cm}Let us now analyze the possible values for $m$ and use Equations~\eqref{n} and \eqref{m} to derive $x_{ij}$ values.
\begin{itemize}[nosep]
	\item If $m=n-1$, then assume $x_{23}> 0$. It follows that $3\leq x_{12}\leq x_{23}$. Hence, $x_{12}=x_{23}=3$, else $G$ is $(3,3)$-dominated. We therefore have $x_{22}=m-6$ and $x_{33}=0$ and $f(\Path{n})-f(G)=-c_{12}+4c_{22}-3c_{23}=4V_6-V_5-V_7=2V_6>0$, which contradicts the maximality of $G$. Hence $x_{23}=0$ which implies $G=\Path{n}\in F_7$.
\item 	if $m=n$, then $x_{12}=x_{23}=0$, else $G$ is $(0,0)$-dominated. We therefore have $x_{22}=m$ and $x_{33}=0$ which implies $G=\C{n}\in F_7$.
\item If $m=n+1$, then
\begin{itemize}
	\item if $x_{12}=0$, then $x_{23}\geq 4$. Hence, $x_{23}=4$ else $G$ is $(0,4)$-dominated.  We therefore have $x_{22}=m-5$ and $x_{33}=1$, meaning that $G\in F_7$.
	\item if $x_{12}=1$, then $x_{23}\geq 3$. Hence, $x_{23}=3$ else $G$ is $(1,3)$-dominated.  We therefore have $x_{22}=m-7$ and $x_{33}=3$, meaning that $G\in F_7$.
	\item if $x_{12}=2$, then $x_{23}\geq 2$. Hence, $x_{23}=2$ else $G$ is $(2,2)$-dominated.  We therefore have $x_{22}=m-9$ and $x_{33}=5$, meaning that $G\in F_7$.
\item if $x_{12}>2$, then $x_{23}>2$, which means that $G$ is $(2,2)$-dominated, which contradicts the maximality of $G$.
\end{itemize}
	\item If $n+1<m\leq\frac{3n-3}{2}$ then
	\begin{itemize}[nosep]
\item if $x_{12}=0$, then $x_{23}\geq 2$. Hence, $x_{23}=2$, else $G$ is $(0,2)$-dominated. We therefore have $x_{22}=3n-2m-1$ and $x_{33}=3m-3n-1$, meaning that $G\in F_7$.
\item if $x_{12}=1$, then $x_{23}\geq 1$. Hence, $x_{23}=1$, else $G$ is $(1,1)$-dominated. We therefore have $x_{22}=3n-2m-3$ and $x_{33}=3m-3n+1$, meaning that $G\in F_7$.
\item if $x_{12}\geq 2$, then $x_{23}\geq 2$. Hence, $G$ is $(1,1)$-dominated, which contradicts the maximality of $G$.\qedhere
\end{itemize}
\end{itemize}
\end{proof}

\begin{cor}\label{cor_F67}
	$F_6\cup F_7$ is the set of extremal chemical graphs for the topological index Sigma.
\end{cor}
\begin{proof}
	It is easy to check that $V_1>0$, $V_2>0$ and $V_5=0$ for Sigma, while $V_3>0$, $V_6>0$ and $V_5+V_7=2V_6$ for \textoverline{Sigma}, which means that $F_6$ is the set of chemical graphs which maximize Sigma, while $F_7$ is the set of chemical graphs which minimize Sigma.
\end{proof}

\begin{defn}\label{familyF8}
		$F_8$ is the set of chemical graphs with the following numbers $x_{ij}$ of $(i,j)$-edges:

\begin{center}
			\setlength{\extrarowheight}{1pt}
		\begin{tabular}{| c | c | c | c |c | l}\cline{1-5} 
			$x_{12}$ & $x_{13}$ & $x_{22}$& $x_{23}$& $x_{33}$&\\ \cline{1-5}
			2 & 0 & $m-2$ & 0 &0& \multirow{2}{*}{if $m+1=n\in\{7,8,9\}$} \\ 
			3 & 0 & $m-6$ & 3 &0& \\ \cline{1-5}
2 & 0 & $m-7$ & 4 &1& if $m=n\in\{7,8\}$  \\ \cline{1-5}
1 & 0 & 1 & 3 &3& if $n=7$ and $m=8$  \\ \cline{1-5}						$\lfloor\frac{3n-2m}{3}\rfloor$ &0 & $m \mymod 3$ & $\lfloor\frac{3n-2m}{3}\rfloor$ &$\lfloor\frac{7m-6n}{3}\rfloor$&otherwise\\ \cline{1-5}
		\end{tabular}
	\end{center}
\end{defn}

\begin{thm}\label{thm_F8}
	Let $f$ be a degree-based topological index such that $V_3{>}0$, $V_6{>}0$ and $V_5{+}V_7{=}4V_6$. A chemical graph $G$ maximizes $f$ over all graphs of the same order and size as $G$ if and only if $G\in F_8$.
\end{thm}
\begin{proof}
Let $G$ be a chemical graph of order $n$, size $m$ and with $x_{ij}$ $(i,j)$-edges. Assume that it maximizes $f$ over all chemical graphs of order $n$ and size  $m$.    	
As shown by Lemma \ref{lem:V3}, $V_3>0$ implies $x_{13}=0$.

If $m+1=n\in\{7,8,9\}$, then there are only two possibilities: either $x_{12}=2$ and $x_{22}=m-2$, or $x_{12}=3$, $x_{23}=3$ and $x_{22}=m-6$. Since $c_{12}-4c_{22}+3c_{23}=V_5+V_7-4V_6=0$, we deduce that both cases correspond to an optimal graph $G$ that belongs to $F_8$. We now assume $m\geq n$ or $m+1=n\geq 10$.

\vspace{0.2cm}\noindent If $n_3=0$, then $G=\Path{n}$ and $m+1=n\geq 10$ or $G=\C{n}$ and $n=m$.
\begin{itemize}[nosep]
      \item If $m+1=n\geq 10$, then the graph $G'$ with $x_{12}=x_{23}=4, x_{33}=1$ and $x_{22}=m-9$ has value $f(G')=f(G)+2c_{12}-7c_{22}+4c_{23}+c_{33}=f(G)+2V_5+2V_7-7V_6=f(G)+V_6>0$, which contradicts the maximality of $G$.
	\item If $n=m$, then the graph $G'$ with $x_{12}=2, x_{23}=4$, $x_{33}=1$ and $x_{22}=m-7$ has value $f(G')=f(G)+2c_{12}-7c_{22}+4c_{23}+c_{33}=f(G)+2V_5+2V_7-7V_6=f(G)+V_6>0$, which contradicts the maximality of $G$.
\end{itemize}

\vspace{0.2cm}\noindent If $n_3>0$, then consider the following 4 cases:
\begin{itemize}[nosep]
      \item If $m=n=7$, then there are only two possibilities: either $x_{12}=2$, $x_{23}=4$ and $x_{33}=1$, or $x_{12}=1$, $x_{23}=3$ and $x_{22}=3$. Since $c_{12}+c_{23}+c_{33}-3c_{22}=V_5+V_7-3V_6=V_6>0$, we deduce that the first solution is the best, which means that $G\in F_8$.
      \item If $m=n=8$, then there are only three possibilities: either $x_{12}=2$, $x_{22}=1$, $x_{23}=4$ and $x_{33}=1$, or $x_{12}=1$, $x_{23}=3$ and $x_{22}=4$, or $x_{12}=2$ and $x_{23}=6$. As in the previous case, the first solution is better than the second. Also, since  $c_{22}-2c_{23}+c_{33}=V_6>0$, we deduce that the first solution is better than the third one, which implies $G\in F_8$.
\item If $n=7$ and $m=8$, there are four possibilities:
\begin{itemize}[nosep]
\item $x_{22}=2$ and $x_{23}=6$;
\item $x_{22}=3$, $x_{23}=4$ and $x_{33}=1$;
\item $ x_{12}=1$, $x_{23}=5$ and $x_{33}=2$;
\item $x_{12}=x_{22}=1$, $x_{23}=x_{33}=3$.
\end{itemize}
The fourth is better than the first since $c_{12}-c_{22}-3c_{23}+3c_{33}=V_5+V_7-V_6=3V_6>0$. It is better than the second since $c_{12}-2c_{22}-c_{23}+2c_{33}=V_5+V_7-2V_6=2V_6>0$. It is better than the third since $c_{22}-2c_{23}+c_{33}=V_6>0$. Hence, $G\in F_8$.
\item For the remaining case where $n\in\{7,8,9\}$ and $m\geq 9$, or $n\geq 10$, consider  the two $(n,m)$-preserving vectors $A=(0,0,1,-2,1)$ and $A'=(1,0,-3,1,1)$ associated with $V_6$ and $V_5+V_7-3V_6$, respectively. Let $(x'_{12},x'_{13},x'_{22},x'_{23},x'_{33})$ be the $(A,\frac{x_{23}-x_{12}}{2})$-transform of $(x_{12},x_{13},x_{22},x_{23},x_{33})$. Note that $x'_{12}=x'_{23}$ and $x'_{13}=0$. Let $(x''_{12},x''_{13},x''_{22},x''_{23},x''_{33})$ be the $(A',\lfloor\frac{x'_{22}}{3}\rfloor)$-transform of $(x'_{12},x'_{13},x'_{22},x'_{23},x'_{33})$. Note that $x''_{13}=0$,
      	$x''_{12}=x''_{23}$ and $x''_{22}\leq 2$. 
      	Equations \eqref{n} and \eqref{m} then give $x''_{22}=m \mymod 3$, $x''_{12}=x''_{23}=\lfloor\frac{3n-2m}{3}\rfloor$, and $x''_{33}=\lfloor\frac{7m-6n}{3}\rfloor$. Consider any graph $G'$ in $F_8$ with $x''_{ij}$ $(i,j)$-edges. We then have
\vspace{-7pt}\begin{center}$f(G')=
    	f(G)+\frac{x_{23}-x_{12}}{2}V_6+\lfloor\frac{x'_{22}}{3}\rfloor(V_5+V_7-3V_6)=f(G)+(\frac{x_{23}-x_{12}}{2}+\lfloor\frac{x'_{22}}{3}\rfloor)V_6.$\end{center}
      	\vspace{-5pt}If $x_{23}-x_{12}>0$, or $x_{23}-x_{12}=0$ and $x_{22}=x'_{22}>2$,then $f(G')>f(G)$, which contradicts the maximality of $G$. Hence, we can choose $G'$ equal to $G$, which implies $G\in F_8$.\qedhere
\end{itemize}
\end{proof}

\begin{defn}\label{familyF9}
		$F_9$ is the set of chemical graphs with the following numbers $x_{ij}$ of $(i,j)$-edges:
\begin{center}
		\setlength{\extrarowheight}{3pt}
		\begin{tabular}{| c | c | c | c |c | l}\cline{1-5} 
			$x_{12}$ & $x_{13}$ & $x_{22}$& $x_{23}$& $x_{33}$&\\ \cline{1-5}
			0 & $\frac{6n-5m+(2m \mymod 3)}{3}$ & 0 & $\frac{8m-6n-4(2m \mymod 3)}{3}$ &$2m \mod 3$& if $n-1\leq m\leq\frac{6n+2}{5}$\\ \cline{1-5}
			0 & 0 & 0 & $6n-4m$ &$5m-6n$& if $\frac{6n+3}{5}\leq m\leq \frac{3n-3}{2}$  \\ \cline{1-5}
		\end{tabular}
	\end{center}
\end{defn}

\begin{thm}\label{thm_F9}
	Let $f$ be a degree-based topological index such that $V_1{>}0$, $V_2{>}0$ and $V_5{<}0$. A chemical graph $G$ maximizes $f$ over all graphs of the same order and size as $G$ if and only if $G\in F_9$.
\end{thm}
\vspace{-15pt}\begin{proof}
Let $G$ be a chemical graph of order $n$, size $m$ and with $x_{ij}$ $(i,j)$-edges. Assume that it maximizes $f$ over all chemical graphs of order $n$ and size  $m$.    	
As shown by Lemma \ref{lem:V2}, $V_1>0$ and $V_2>0$ imply $x_{12}=x_{22}=0$.
	\begin{itemize}[nosep]
		\item If $x_{33}\leq 2$,  Equations~\eqref{n} and \eqref{m} give $x_{33}=(2m \mymod 3)$, $x_{13}=\frac{6n-5m+(2m \mymod 3)}{3}$ and $x_{23}=\frac{8m-6n-4(2m \mymod 3)}{3}$. Since $x_{13}\geq 0$, we have $6n-5m\geq -2$, which implies $G\in F_9$.
		\item If $x_{33}\geq 3$ then
		\begin{itemize}[nosep]
			\item if $x_{13}=0$,  Equations~\eqref{n} and \eqref{m} give $x_{23}=6n-4m$ and $x_{33}=5m-6n$. Since $x_{33}\geq 3$, we have $m\geq \frac{6n+3}{5}$, which implies $G\in F_9$.			
			\item if $x_{13}>0$, consider the $(n,m)$-preserving vector $A=(0,1,0,-4,3)$ associated with $V_5$, let $a=\min\{x_{13},\lfloor\frac{x_{33}}{3}\rfloor\}$, and let  $(x'_{12},x'_{13},x'_{22},x'_{23},x'_{33})$ be the $(A,-a)$-transform of $(x_{12},x_{13},x_{22},x_{23},x_{33})$. Then either $x'_{33}\leq 2$, or $x'_{33}\geq 3$ and $x_{13}=0$. In both cases, we have seen that there is a graph $G'\in F_9$ with $x'_{ij}$ $(i,j)$-edges. We therefore have
			$f(G')=f(G)-aV_5>f(G)$,
			which contradicts the maximality of $G$.\qedhere
			\end{itemize}
		\end{itemize}
			\end{proof}

\begin{cor}\label{cor_F89}
	$F_8\cup F_9$ is the set of extremal chemical graphs for the topological index ABC. 
\end{cor}
\vspace{-15pt}\begin{proof}
	It is easy to check that $V_3>0$, $V_6>0$ and $V_5+V_7=4V_6$ for \textoverline{ABC}, while $V_1>0$, $V_2>0$ and $V_5<0$ for ABC, which means that $F_8$ is the set of chemical graphs which minimize ABC, while $F_9$ is the set of chemical graphs which maximize ABC.
\end{proof}

\begin{thm}\label{thm_F82}
	A chemical graph  $G$ maximizes the $f$=aZagreb topological index over all graphs of the same order and size as $G$ if and only if $G{\in }F_{8}$, except in two cases where the $x_{ij}$ values of $G$ are as follows: if $n{=}7$ and $m{=}8$  then $x_{12}{=}x_{13}{=}x_{23}{=}1$, $x_{22}=0$ and $x_{33}{=}5$;  if $n{=}8$ and $m{=}8$  then  $x_{12}{=}x_{23}{=}2$, $x_{13}{=}1$, $x_{22}=0$ and $x_{33}{=}3$.
\end{thm}
\vspace{-15pt}\begin{proof}
   The aZagreb topological index is defined by $c_{ij}{=}(\frac{ij}{i{+}j{-}2})^3$ (see Table \ref{tab_inv}). 
Let $G$ be a chemical graph of order $n $, size $m$ and with $x_{ij}$ $(i,j)$-edges. Assume that it maximizes $f$ over all chemical graphs of order $n$ and size  $m$.  

Consider the three $(n,m)$-preserving vectors $A_1,A_2,A_3$ associated with the three  following strictly positive values $W_1,W_2,W_3$:

\vspace{-10pt} \begin{itemize}
       \setlength{\itemsep}{-5pt}
 \item $A_1=(1,-1,0,-1,1)$ is associated with $W_1=c_{12}-c_{13}-c_{23}+c_{33}\approx 8.01$;
 \item $A_2=(1,-1,-1,1,0)$ is associated with $W_2=c_{12}-c_{13}-c_{22}+c_{23}\approx 4.62$;
 \item $A_3=(2,-3,0,2,-1)$ is associated with $W_3=2c_{12}-3c_{13}+2c_{23}-c_{33}\approx 10.48$.
 \end{itemize}
\vspace{-10pt} Let 
\vspace{-10pt}  \begin{itemize}
       \setlength{\itemsep}{-4pt}
     \item $(x^1_{12}, x^1_{13}, x^1_{22}, x^1_{23}, x^1_{33})$ be the $(A_1,\max\left\{0,\min\{x_{13},\frac{x_{23}-x_{12}}{2}\}\right\})$-transform of $(x_{12}$, $ x_{13}$, $x_{22}$, $x_{23}$,  $x_{33})$;
     \item $(x^2_{12},x^2_{13},x^2_{22},x^2_{23},x^2_{33})$ be the $(A_2,\min\{x^1_{13},x^1_{22}\})$-transform of $(x^1_{12},x^1_{13},x^1_{22},x^1_{23},x^1_{33})$;
     \item $(x^3_{12},x^3_{13},x^3_{22},x^3_{23},x^3_{33})$ be the $(A_3,\lfloor\frac{x^2_{13}}{3}\rfloor)$-transform of $(x^2_{12},x^2_{13},x^2_{22},x^2_{23},x^2_{33})$.
     \end{itemize}
Note that if $x^2_{13}>0$, then $x^2_{22}=0$ and $x^2_{12}=x^2_{23}$, which implies $x^2_{33}>0$, else $G$ has order $n\leq 6$. We then have $x^3_{13}\leq 2$ and if $x^3_{13}=0$, then $x^3_{12}=x^3_{23}$ and $x^3_{22}=0$. There are therefore only 3 possible cases for which we can derive the $x^3_{ij}$ values using Equations~\eqref{n} and \eqref{m}:
\begin{itemize}[nosep]
    \item[(1)] if $x^3_{13}=0$, then $G\in F_8$. Indeed the proof of Theorem \ref{thm_F8} uses the fact that $V_3>0$ only to show that $x_{13}=0$, and it is easy to check that $V_6>0$ and $V_5+V_7=4V_6$ for the aZagreb topological index. Therefore, 
    \begin{itemize}[nosep]
    \item if $x_{13}=0$, then all $x_{ij}$ values are equal to the $x^3_{ij}$ values and as in Theorem \ref{thm_F8}, we conclude that $G\in F_8$;
    \item if $x_{13}>0$, then as in Theorem \ref{thm_F8}, we know that there exists a graph $G'\in F_8$ of order $n$ and size $m$ which contradicts the maximality of $G$ since
\vspace{-0.2cm}\begin{eqnarray*}f(G')&\geq &f(x^3_{12},x^3_{13},x^3_{22},x^3_{23},x^3_{33})\\[0pt]&=&f(G){+}W_1(\max\left\{0,\min\{x_{13},\frac{x_{23}-x_{12}}{2}\}\right\}){+}W_2(\min\{x^1_{13},x^1_{22}\}){+}W_3(\lfloor\frac{x^2_{13}}{3}\rfloor)\\[-10pt]&>&f(G).
\end{eqnarray*}
    \end{itemize}
    \item[(2)] if $x^3_{13}=1$, $x^3_{12}=x^3_{23}$ and $x^3_{22}$=0, then Equations~\eqref{n} and \eqref{m} give $x^3_{12}=x^3_{23}=\frac{3n-2m-2}{3}$ and $x^3_{33}=\frac{7m-6n+1}{3}$, which implies $m \mymod 3=2$.
        \item[(3)] if $x^3_{13}=2$, $x^3_{12}=x^3_{23}$ and $x^3_{22}$=0, then Equations~\eqref{n} and \eqref{m} give $x^3_{12}=x^3_{23}=\frac{3n-2m-4}{3}$ and $x^3_{33}=\frac{7m-6n+2}{3}$, which implies $m \mymod 3=1$.
\end{itemize}
Let us analyze the situation according to the value of $m\mymod 3$:
\begin{itemize}[nosep]
\item if $m\mymod 3=0$ then $G\in F_8$ (since we are in Case (1));
\item if $m\mymod 3=1$, then 
\begin{itemize}[nosep]
\item if $m \geq 10$ or $m+1=n=8$, then Case (1) is better than Case (3) since $c_{12}-2c_{13}+c_{22}+c_{23}-c_{33}\approx 5.85>0$. Hence, $G\in F_8.$
\item if $m=n=7$, then Case (1) is better than Case (3) since $c_{12}-2c_{13}+3c_{23}-2c_{33}\approx 2.46>0.$ Hence $G\in F_8.$
\end{itemize}
\item if $m\mymod 3=2$, 
\begin{itemize}[nosep]
    \item if $m\geq 11$ or $m+1=n=9$ then Case (1) is better than Case (2) since $-c_{13}+2c_{22}-c_{33}\approx 1.23>0.$ Hence $G\in F_8.$
    \item if $m=8$ and $n\in\{7,8\}$ then Case (2) is better than Case (1) since $c_{13}-c_{22}-2c_{23}+2c_{33}\approx 2.15>0.$ Moreover, the $x_{ij}$ values of $G$ are equal to the $x^3_{ij}$ values else the graph $G'$ with $x^3_{ij}$ $(i,j)$-edges is such that $f(G')>f(G)$. Hence, 
    $x_{12}=x_{13}=x_{23}=1$ and $x_{33}=5$ if $n=7$ and $x_{12}=x_{23}=2$, $x_{13}=1$ and $x_{33}=3$ if $n=8$.\qedhere
\end{itemize}
\end{itemize}

\end{proof}

\vspace{0.1cm}\begin{defn}\label{familyF10}
		$F_{10}$ is the set of chemical graphs with the following numbers $x_{ij}$ of $(i,j)$-edges:
\begin{center}
		\setlength{\extrarowheight}{4pt}
		\begin{tabular}{| c | c | c | c |c | l}\cline{1-5} 
			$x_{12}$ & $x_{13}$ & $x_{22}$& $x_{23}$& $x_{33}$&\\ \cline{1-5}
			0 & $\frac{6n-5m}{3}$ & 0 & $\frac{8m-6n}{3}$ &0& if $n-1\leq m\leq\frac{6n-2}{5}$ and $m\mymod 3=0$\\ \cline{1-5}
			0 & $\frac{6n-5m-1}{3}$ & 1 & $\frac{8m-6n-2}{3}$ &0& if $n-1\leq m\leq\frac{6n-2}{5}$ and $m\mymod 3=1$\\ \cline{1-5}
			0 & $\frac{6n-5m+1}{3}$ & 0 & $\frac{8m-6n-4}{3}$ &1& if $n-1\leq m\leq\frac{6n-2}{5}$ and $m\mymod 3=2$\\ \cline{1-5}
			0 & 0 & 1 & $m-1$ &0& if $m=\frac{6n-1}{5}$\\ \cline{1-5}
			0 & 0 & 0 & $6n-4m$ &$5m-6n$& if $\frac{6n}{5}\leq m\leq \frac{3n-3}{2}$\\ \cline{1-5}
		\end{tabular}
	\end{center}
\end{defn}

\vspace{0.12cm}\begin{thm}\label{thm_F10}
	A chemical graph $G$ minimizes the $f$=aZagreb topological index over all graphs of the same order and size as $G$ if and only if $G\in F_{10}$.
\end{thm}
\begin{proof}
As in the previous theorem, we have $c_{ij}{=}(\frac{ij}{i{+}j{-}2})^3$. Let $G$ be a chemical graph of order $n$, size $m$ and with $x_{ij}$ $(i,j)$-edges. Assume that it minimizes $f$ over all chemical graphs of order $n$ and size  $m$.    	
Note that $G\neq \Path{n}$ since for $n=m-1$, the chemical graph $G'$ with $x'_{12}=1$, $x'_{13}=2$, $x'_{23}=1$ and $x'_{22}=m-4$ would have value $f(G')=f(G)-c_{12}+2c_{13}-2c_{22}+c_{23}\approx f(G)-9.25$. Hence, $x_{12}\leq x_{23}$.

Consider the five $(n,m)$-preserving vectors $A_1,A_2,A_3,A_4,A_5$ associated with the five  following strictly negative values $W_1,W_2,W_3,W_4,W_5$:

\vspace{-6pt} \begin{itemize}
       \setlength{\itemsep}{-4pt}
 \item $A_1=(-1,1,1,-1,0)$ is associated with $W_1=-c_{12}+c_{13}+c_{22}-c_{23}\approx -4.62$;
 \item $A_2=(0,1,-2,0,1)$ is associated with $W_2=c_{13}-2c_{22}+c_{33}\approx -1.23$;
 \item $A_3=(0,-1,0,4,-3)$ is associated with $W_3=-c_{13}+4c_{23}-3c_{33}\approx -5.54$;
 \item $A_4=(0,0,-1,2,-1)$ is associated with $W_4=-c_{22}+2c_{23}-c_{33}\approx -3.39$;
  \item $A_5=(0,-1,1,2,-2)$ is associated with $W_5=-c_{13}+c_{22}+2c_{23}-2c_{33}\approx -2.15$.
\end{itemize}
\vspace{-10pt} Let 
\vspace{-10pt}  \begin{itemize}
       \setlength{\itemsep}{-4pt}
     \item $(x^1_{12},x^1_{13},x^1_{22},x^1_{23},x^1_{33})$ be the $(A_1,x_{12})$-transform of $(x_{12},x_{13},x_{22},x_{23},x_{33})$;
     \item $(x^2_{12},x^2_{13},x^2_{22},x^2_{23},x^2_{33})$ be the $(A_2,\lfloor\frac{x^1_{22}}{2}\rfloor)$-transform of $(x^1_{12},x^1_{13},x^1_{22},x^1_{23},x^1_{33})$;
     \item $(x^3_{12},x^3_{13},x^3_{22},x^3_{23},x^3_{33})$ be the $(A_3,\min\{\lfloor\frac{x^2_{33}}{3}\rfloor,x^2_{13}\})$-transform of $(x^2_{12},x^2_{13},x^2_{22},x^2_{23},x^2_{33})$;
     \item $(x^4_{12},x^4_{13},x^4_{22},x^4_{23},x^4_{33})$ be the $(A_4,\min\{x^3_{22},x^3_{33}\})$-transform of $(x^3_{12},x^3_{13},x^3_{22},x^3_{23},x^3_{33})$;
     \item $(x^5_{12},x^5_{13},x^5_{22},x^5_{23},x^5_{33})$ be the $(A_5,\min\{x^4_{13},\lfloor\frac{x^4_{33}}{2}\rfloor\})$-transform of $(x^4_{12},x^4_{13},x^4_{22},x^4_{23},x^4_{33})$.
     \end{itemize}
We then have $x^5_{12}=0, x^5_{22}\leq 1$, $x^5_{13}=0$ or $x^5_{33}\leq 1$, and $x^4_{22}+x^4_{33}\leq 1$. There are therefore only 5 possible cases for which we can derive the $x^5_{ij}$ values using Equations~\eqref{n} and \eqref{m}:
\begin{itemize}
	       \setlength{\itemsep}{-4pt}
\item if $x^5_{22}=x^5_{33}=0$ and $x^5_{13}\geq 1$, then $x^5_{13}=\frac{6n-5m}{3}$ and $x^5_{23}=\frac{8m-6n}{3}$, which implies $6n-5m\geq 3$ and $m\mymod 3=0$;
\item if $x^5_{22}=1, x^5_{33}=0$ and $x^5_{13}\geq 1$, then $x^5_{13}=\frac{6n-5m-1}{3}$ and $x^5_{23}=\frac{8m-6n-2}{3}$, which implies $6n-5m\geq 4$ and $m\mymod 3=1$;
\item if $x^5_{22}=0, x^5_{33}=1$ and $x^5_{13}\geq 1$, then $x^5_{13}=\frac{6n-5m+1}{3}$ and $x^5_{23}=\frac{8m-6n-4}{3}$, which implies $6n-5m\geq 2$ and $m\mymod 3=2$;
\item if $x^5_{22}=x^5_{13}=0$ and $x^5_{33}\geq 0$, then $x^5_{33}=5m-6n$ and $x^5_{23}=6n-4m$, which implies $6n-5m\leq 0$;
\item if $x^5_{22}=1$ and $x^5_{13}=x^5_{33}= 0$, then $x^5_{33}=5m-6n+1$ and $x^5_{23}=6n-4m-2$, which implies $6n-5m=1$.
\end{itemize}
Hence all possible $x^5_{ij}$ values correspond to those in $F_{10}$. So, let $G'$ be a graph with $x^5_{ij}$ $(i,j)$-edges. We have
\vspace{-0.0cm}$$f(G')=f(G)+x_{12}W_1+\lfloor\frac{x^1_{22}}{2}\rfloor W_2+\min\{\lfloor\frac{x^2_{33}}{3}\rfloor,x^2_{13}\}W_3+\min\{x^3_{22},x^3_{33}\}W_4+\min\{x^4_{13},\lfloor\frac{x^4_{33}}{2}\rfloor\}W_5.
\vspace{-0.0cm}$$
If $G$ does not belong to $F_{10}$ then at least one of the five values $x_{12}$, $\lfloor\frac{x^1_{22}}{2}\rfloor$, $\min\{\lfloor\frac{x^2_{33}}{3}\rfloor,x^2_{13}\}$, $\min\{x^3_{22},x^3_{33}\}$ and $\min\{x^4_{13},\lfloor\frac{x^4_{33}}{2}\rfloor\}$ is strictly positive, which implies $f(G')<f(G)$, a contradiction.
\end{proof}

\vspace{0.2cm}\begin{defn}\label{familyF11}
	Family $F_{11}$ is obtained from $F_{10}$ by adding the following possible values:\begin{center}
		\setlength{\extrarowheight}{4pt}
		\begin{tabular}{| c | c | c | c |c | l}\cline{1-5} 
			$x_{12}$ & $x_{13}$ & $x_{22}$& $x_{23}$& $x_{33}$&\\ \cline{1-5}
			1 & $\frac{6n-5m-4}{3}$ & 0 & $\frac{8m-6n+1}{3}$ &0& \multirow{2}{*}{if $n-1\leq m\leq\frac{6n-2}{5}$ and $m\mymod 3=1$}\\ 
			0 & $\frac{6n-5m+2}{3}$ & 0 & $\frac{8m-6n-8}{3}$ &2& \\ \cline{1-5}
			0 & 1 & 0 & $m-3$ &2& if $m=\frac{6n-1}{5}$\\ \cline{1-5}
		\end{tabular}
	\end{center}
\end{defn}

\begin{thm}\label{thm_F11}
	A chemical graph $G$ maximizes the $f$=Albertson topological index over all graphs of the same order and size as $G$ if and only if $G\in F_{11}$.
\end{thm}
\vspace{-15pt}\begin{proof}
   The Albertson topological index is defined by  $c_{ij}=|i-j|$ (see Table \ref{tab_inv}). Let $G$ be a chemical graph of order $n$, size $m$ and with $x_{ij}$ $(i,j)$-edges. Assume that it maximizes $f$ over all chemical graphs of order $n$ and size  $m$.  Note that the two possible cases for $m=\frac{6n-1}{5}$ have the same value since $c_{13}-c_{22}-2c_{23}+2c_{33}=0$. Also, the case in $F_{10}$ for $m\leq\frac{6n-2}{5}$ and $m\mymod 3=1$ has the same value as the two new possibilities in $F_{11}$ since $c_{12}-c_{13}-c_{22}+c_{23}=c_{13}-c_{22}-2c_{23}+2c_{33}=0$. Moreover, $G\neq \Path{n}$ since for $n=m-1$, the chemical graph $G'$ with $x'_{12}=1$, $x'_{13}=2$, $x'_{23}=1$ and $x'_{22}=m-4$ would have value $f(G')=f(G)-c_{12}+2c_{13}-2c_{22}+c_{23}=f(G)+4$. Hence, $x_{12}\leq x_{23}$. 	
   
  \vspace{0.1cm}\noindent Consider the five following $(n,m)$-preserving vectors $A_1,A_2,A_3,A_4,A_5,A_6$:
\vspace{-0.3cm}	\begin{itemize}
       \setlength{\itemsep}{-5pt}
			\item $A_1=(-2,3,0,-2,1)$ is associated with $-2c_{12}+3c_{13}-2c_{23}+c_{33}=2$;
	\item $A_2=(-1,2,-1,-1,1)$ is associated with $-c_{12}+2c_{13}-c_{22}-c_{23}+c_{33}=2$;
		\item $A_3=(0,1,-2,0,1)$ is associated with $c_{13}-2c_{22}+c_{33}=2$;
		\item $A_4=(0,0,-1,2,-1)$ is associated with $-c_{22}+2c_{23}-c_{33}=2$;
		\item $A_5=(-1,1,0,1,-1)$ is associated with $-c_{12}+c_{13}+c_{23}-c_{33}=2$;
		\item $A_6=(0,-1,0,4,-3)$ is associated with $-c_{13}+4c_{23}-3c_{33}=2$.
	\end{itemize}
\vspace{-0.4cm}	Let 
\vspace{-0.5cm}	\begin{itemize}
       \setlength{\itemsep}{-4pt}
			\item $(x^1_{12},x^1_{13},x^1_{22},x^1_{23},x^1_{33})$ be the $(A_1,\lfloor\frac{x_{12}}{2}\rfloor)$-transform of $(x_{12},x_{13},x_{22},x_{23},x_{33})$;
            \item $(x^2_{12},x^2_{13},x^2_{22},x^2_{23},x^2_{33})$ be the $(A_2,\min\{x^1_{12},x^1_{22}\})$-transform of $(x^1_{12},x^1_{13},x^1_{22},x^1_{23},x^1_{33})$;
		\item $(x^3_{12},x^3_{13},x^3_{22},x^3_{23},x^3_{33})$ be the $(A_3,\lfloor\frac{x^2_{22}}{2}\rfloor)$-transform of $(x^2_{12},x^2_{13},x^2_{22},x^2_{23},x^2_{33})$;
		\item $(x^4_{12},x^4_{13},x^4_{22},x^4_{23},x^4_{33})$ be the $(A_4,\min\{x^3_{22},x^3_{33}\})$-transform of $(x^3_{12},x^3_{13},x^3_{22},x^3_{23},x^3_{33})$;
		\item $(x^5_{12},x^5_{13},x^5_{22},x^5_{23},x^5_{33})$ be the $(A_5,\min\{x^4_{12},x^4_{33}\})$-transform of $(x^4_{12},x^4_{13},x^4_{22},x^4_{23},x^4_{33})$;
		\item $(x^6_{12},x^6_{13},x^6_{22},x^6_{23},x^6_{33})$ be the $(A_6,\min\{x^5_{13},\lfloor\frac{x^5_{33}}{3}\rfloor\})$-transform of $(x^5_{12},x^5_{13},x^5_{22},x^5_{23},x^5_{33})$.
	\end{itemize}
We then have $x^6_{12}+x^6_{22}\leq 1$, $\min\{x^6_{12},x^6_{33}\}{=}0$,  $\min\{x^6_{22},x^6_{33}\}{=}0$, and $x^6_{13}=0$ or $x^6_{33}\leq 2$. Hence, there are 6 possible cases for which we can derive the $x^6_{ij}$ values  using Equations~\eqref{n} and \eqref{m}:
\vspace{-0.2cm}	
    \begin{itemize}
       \setlength{\itemsep}{-4pt}
		\item if $x^6_{12}=0$, $x^6_{22}=1$ and $x^6_{33}=0$, then $x^6_{13}=\frac{6n-5m-1}{3}$ and $x^6_{23}=\frac{8m-6n-2}{3}$, which implies $6n-5m\geq 1$ and $m\mymod 3=1$;
		\item if $x^6_{12}=1, x^6_{22}=0$ and $x^6_{33}=0$, then  $x^6_{13}=\frac{6n-5m-4}{3}$ and $x^6_{23}=\frac{8m-6n+1}{3}$, which implies $6n-5m\geq 4$ and $m\mymod 3=1$;
		\item if $x^6_{12}=x^6_{22}=x^6_{13}=0$ then $x^6_{23}=6n-4m$ and $x^6_{33}=5m-6n$, which implies $6n-5m\leq 0$;
		\item if $x^6_{12}=x^6_{22}=x^6_{33}=0$ and  $x^6_{13}\geq 1$, then  $x^6_{13}=\frac{6n-5m}{3}$ and $x^6_{23}=\frac{8m-6n}{3}$, which implies $6n-5m\geq 3$ and $m\mymod 3=0$;
		\item if $x^6_{12}=x^6_{22}=0$, $x^6_{33}=1$ and  $x^6_{13}\geq 1$, then $x^6_{13}=\frac{6n-5m+1}{3}$ and $x^6_{23}=\frac{8m-6n-4}{3}$, which implies $6n-5m\geq 2$ and $m\mymod 3=2$;
		\item if $x^6_{12}=x^6_{22}=0$, $x^6_{33}=2$ and  $x^6_{13}\geq 1$, then $x^6_{13}=\frac{6n-5m+2}{3}$ and $x^6_{23}=\frac{8m-6n-8}{3}$, which implies $6n-5m\geq 1$ and $m\mymod 3=1$.
			\end{itemize}
Hence all possible $x^6_{ij}$ values correspond to those in $F_{11}$. So, let $G'$ be a graph with $x^6_{ij}$ $(i,j)$-edges. We have
\vspace{-0.3cm}\begin{center}$f(G'){=}f(G){+}2\left(\lfloor\frac{x_{12}}{2}\rfloor
    {+}\min\{x^1_{12},x^1_{22}\}
	{+}\lfloor\frac{x^2_{22}}{2}\rfloor
	{+}\min\{x^3_{22},x^3_{33}\}
	{+}\min\{x^4_{12},x^4_{33}\}
	{+}\min\{x^5_{13},\lfloor\frac{x^5_{33}}{3}\rfloor\}\right)$\end{center}
\vspace{-0.1cm}	If $G$ does not belong to $F_{11}$ then at least one of the six values $\lfloor\frac{x_{12}}{2}\rfloor$, $\min\{x^1_{12},x^1_{22}\}$, $\lfloor\frac{x^2_{22}}{2}\rfloor$, 
$\min\{x^3_{22},x^3_{33}\}$, 
$\min\{x^4_{12},x^4_{33}\}$,
$\min\{x^5_{13},\lfloor\frac{x^5_{33}}{3}\rfloor\}$
is strictly positive, which implies $f(G')>f(G)$, a contradiction.\qedhere
\end{proof}

\begin{defn}\label{familyF12}
		$F_{12}$ is the set of chemical graphs with the following numbers $x_{ij}$ of $(i,j)$-edges:
\begin{center}
	\setlength{\extrarowheight}{0pt}
		\begin{tabular}{c| c | c | c | c |c | l}\cline{2-6} 
			&$x_{12}$ & $x_{13}$ & $x_{22}$& $x_{23}$& $x_{33}$&\\ \cline{2-6}
			&2 & 0 & $m-2$ & 0 &0& if $m=n-1$\\ \cline{2-6}
			&0 & 0 & $m$ & 0 &0& if $m=n$\\ \cline{2-6}
			$(n\geq 8)$&2 & 0 & $m-9$ & 2 &5& \multirow{5}{*}{if $m=n+1$}\\ 
			&1 & 1 & $m-8$ & 1 &5& \\
			&1 & 0 & $m-7$ & 3 &3& \\
			&0 & 1 & $m-6$ & 2 &3& \\
			&0 & 0 & $m-5$ & 4 &1& \\\cline{2-6}			
		&	0 & 0 & $3n-2m-1$ &2 &$3m-3n-1$& \multirow{2}{*}{if $n+1<m\leq \frac{3n-3}{2}$}\\
			&1 & 0 & $3n-2m-3$ &1 &$3m-3n+1$&\\			 \cline{2-6}
		\end{tabular}
	\end{center}
\end{defn}

\begin{thm}\label{thmF12}
	A chemical graph $G$ minimizes the $f$=Albertson topological index over all graphs of the same order and size as $G$ if and only if $G\in F_{12}$.
\end{thm}
\begin{proof}
As in the previous theorem, we have $c_{ij}=|i-j|$. Let $G$ be a chemical graph of order $n$, size $m$ and with $x_{ij}$ $(i,j)$-edges. Assume that it minimizes $f$ over all chemical graphs of order $n$ and size  $m$.  
If $m=n-1$, then $G=\Path{n}$ since $f(\Path{n})=2$ while $f(G)\geq 6$ if $n_3>0$. Also, if $m=n$, then $G=\C{n}$ since $f(\C{n})=0$ while $f(G)>0$ if $n_1+n_3>0$ and $m<\frac{3n}{2}$. Hence, in these cases, we have $G\in F_{12}$.

Assume now $m\geq n+1$. We thus have $x_{12}\leq x_{23}$ and $f(G)$ is an even number at least equal to 2 (since $m<\frac{3n}{2}).$ To reach value 2, there are only three possibilities:
\vspace{-0.3cm}\begin{itemize}
       \setlength{\itemsep}{-4pt}
	\item if $x_{13}=1$ and $x_{12}=x_{23}=0$, then $x_{22}=0$ and Equations~\eqref{n} and \eqref{m} give $x_{22}=3n-2m-2$ and $x_{33}=3m-3n+1$ which implies $m=\frac{3n-2}{2}$, a contradiction.
	\item if $x_{12}=x_{23}=1$ and $x_{13}=0$, then Equations~\eqref{n} and \eqref{m} give $x_{22}=3n-2m-3$ and $x_{33}=3m-3n+1$, which implies $m>n+1$ (else $x_{33}=4$ and $n_3=3$ imply $x_{33}=4>3=\frac{n_3((n_3-1)}{2})$) and $G\in F_{12}$.
	\item if $x_{12}=0$, $x_{23}=2$ and $x_{13}=0$, then Equations~\eqref{n} and \eqref{m} give $x_{22}=3n-2m-1$ and $x_{33}=3m-3n-1$, which implies $m>n+1$ (else $x_{33}=1$ and $n_3=1$ imply $x_{33}=1>0=\frac{n_3((n_3-1)}{2})$) and $G\in F_{12}$.
\end{itemize}
Hence, if $m>n+1$ then $G\in F_{12}$. The remaining case is $m=n+1$ for which $f(G)\geq 4$. There are only six possibilities to reach the minimum value 4:
\vspace{-0.3cm}\begin{itemize}
       \setlength{\itemsep}{-4pt}
	\item if $x_{12}=x_{23}=2$ and $x_{13}=0$, then $x_{22}=m-9$ and $x_{33}=5$.
	\item if $x_{12}=x_{23}=1$ and $x_{13}=1$, then $x_{22}=m-8$ and $x_{33}=5$.
	\item if $x_{12}=1$, $x_{23}=3$ and $x_{13}=0$, then $x_{22}=m-7$ and $x_{33}=3$.
	\item if $x_{12}=0$, $x_{23}=2$ and $x_{13}=1$, then $x_{22}=m-6$ and $x_{33}=3$.
	\item if $x_{12}=0$, $x_{23}=4$ and $x_{13}=0$, then $x_{22}=m-5$ and $x_{33}=1$.
	\item if $x_{12}=x_{23}=0$ and $x_{13}=2$, then $x_{22}=0$ and $x_{33}=m-2$, which implies $n=6$.
\end{itemize}
\vspace{-0.2cm}Since $G$ is of order $n\geq 7$, we have $G\in F_{12}$.
\end{proof}

\section{Conclusion}

Many topological indices have been proposed to study the chemical properties of molecules, and many papers focus on extremal graphs for these indices, each paper dealing with a particular index. We have shown that many of these topological indices have the same extremal properties in the sense that the chemical graphs that maximize or minimize the values of these indices are often the same. Thus, for example, for 29 of these indices, one might expect 58 classes of extremal chemical graphs, while 5 families are sufficient to describe them all. Also, for another example, chemical graphs of even order $n$ for which $x_{13}=\frac{3n-2m}{2}$, $x_{33}=\frac{4m-3n}{2}$ and $x_{12}=x_{22}=x_{23}=0$  are extremal for 29 topological indices (since these graphs belong to $F_1 \cap F_3$). 

Most of the characterizations we have given for extremal graphs are based on a set of 8 values $V_1,\ldots,V_8$. If new topological indices are proposed, it is therefore easy to check whether they have the same extremal properties of the indices studied in this paper. Note that some degree-based topological indices that we have not analyzed in this paper do not have any of the stated properties that allow us to characterize their extremal chemical graphs. For example, the reduced reciprocal Randi\'c index (rrRandi\'c) mentioned in \cite{GFE14} and defined by $c_{ij}=\sqrt{(i-1)(j-1)}$ is such that $V_1, V_2, V_3$ and $V_4$ are strictly negative. An analysis similar to those performed in Section \ref{secFamilies} easily shows that the set of extremal chemical graphs of order $n\geq 10$ for the aZagreb index is strictly contained in that for rrRandi\'c.

As Ivan Gutman pointed out \cite{G13},  ``today we have far too many
topological indices, and there seems to lack a firm
criterion to stop or slow down their proliferation''.  We believe we have provided a tool to quickly test whether a new topological index has the same extremal properties as many existing indices.

\begin{spacing}{0.1}

\end{spacing}

\end{document}